\documentclass[11pt]{amsart}
\usepackage[usenames]{color}
\usepackage{amsmath,amssymb,amsthm,amsfonts}
\usepackage[colorlinks=true,
linkcolor=webgreen,
filecolor=webbrown,
citecolor=webgreen]{hyperref}
\usepackage{graphicx}
\usepackage{epstopdf}
\usepackage{url}
\usepackage{comment}
\usepackage[normalem]{ulem}
\usepackage{tikz}

\definecolor{webgreen}{rgb}{0,.5,0}
\definecolor{webbrown}{rgb}{.6,0,0}

\newtheorem{theorem}{Theorem}[section]
\newtheorem{lemma}[theorem]{Lemma}
\newtheorem{prop}[theorem]{Proposition}
\newtheorem*{prop*}{Proposition}
\newtheorem{cor}[theorem]{Corollary}
\newtheorem*{cor*}{Corollary}
\newtheorem{corollary}[theorem]{Corollary}

\theoremstyle{definition}
\newtheorem{definition}[theorem]{Definition}
\newtheorem{example}[theorem]{Example}

\newtheorem{remark}[theorem]{Remark}
\newtheorem*{remark*}{Remark}
\newtheorem*{example*}{Example}

\newcommand{\thistheoremname}{}
\newtheorem*{genericthm*}{\thistheoremname}
\newenvironment{namedthm*}[1]
{\renewcommand{\thistheoremname}{#1}
\begin{genericthm*}}
{\end{genericthm*}}

\newcommand{\lk}{{\rm lk}}

\numberwithin{equation}{section}
\title{Ample Simplicial Complexes}

\author {Chaim Even-Zohar}
\address{Alan Turing Institute\\
 London, UK}
\email{cevenzohar@turing.ac.uk}
\thanks{C. Even-Zohar was partially supported by the Lloyd Register Foundation and the Data Centrick Engineering Programme of the Alan Turing Institute}

\author{Michael Farber}
\address{School of Mathematical Sciences \\
Queen Mary University of London\\
London, E1 4NS, 
United Kingdom}
\email{m.farber@qmul.ac.uk}
\thanks{M. Farber was partially supported by a grant from the Leverhulme Foundation}

\author {Lewis Mead}
\address{School of Mathematical Sciences \\
Queen Mary University of London\\
London, E1 4NS, 
United Kingdom}
\email{lewis.mead@qmul.ac.uk}
\date{\today}

\begin{document}
\maketitle
\begin{abstract}
Motivated by potential applications in network theory, engineering and computer science, we study $r$-ample simplicial complexes. These complexes can be viewed as finite approximations to the Rado complex which has a remarkable property of {\it indestructibility,}  in the sense that removing any finite number of its simplexes leaves a complex isomorphic to itself. 
We prove that an $r$-ample simplicial complex is simply connected and $2$-connected for $r$ large. 
The number $n$ of vertexes of an $r$-ample simplicial complex satisfies $\exp(\Omega(\frac{2^r}{\sqrt{r}}))$.
We use the probabilistic method to establish the existence of $r$-ample simplicial complexes with $n$ vertexes for any 
 $n>r 2^r 2^{2^r}$. 
Finally, we introduce {\it the iterated Paley simplicial complexes}, which are explicitly constructed $r$-ample simplicial complexes with nearly optimal number of vertexes. 

\end{abstract}

\section{Introduction}

Modern network science uses simplicial complexes of high dimension for modelling complex networks consisting of a vast number of interacting entities: while pairwise interactions can be recorded by a graph, the higher order interactions between the entities require the language of simplicial complexes of dimension greater than $1$. We refer the reader to a recent survey ~\cite{battison} with many interesting references including
applications to social systems, neuroscience, ecology, and to 
biological sciences.

Viewing simplicial complexes as representing networks raises new interesting 
questions about their geometry and topology.  
Motivated by this viewpoint, we study in this paper a special class of simplicial complexes representing {\lq\lq stable and resilient\rq\rq} networks, in the sense that small alterations of the network have limited impact on its global properties (such as connectivity and high connectivity). 

In \cite{farber2019rado} the authors investigated a remarkable simplicial complex $X$ (called {\it the Rado complex}) which is 
{ \lq\lq totally  indestructible\rq\rq } in the following sense: 
removing any finite number of simplexes of $X$ leaves a simplicial complex {\it isomorphic} to $X$. 
Unfortunately, $X$ has infinite (countable) number of vertexes and cannot be practically implemented. 
In the present paper we study finite simplicial complexes which can be viewed as approximations to the Rado complex $X$. We call such complexes $r$-ample, where $r\ge 1$ is an integer characterising the depth or level of ampleness. 
The Rado complex is the only simplicial complex on countably many vertexes which is $\infty$-ample.  The finite simplicial complexes which we study in this paper have a limited amount of ampleness and a limited amount of indestructibility. 
The formal definition of $r$-ampleness requires the existence of all possible extensions of simplicial subcomplexes of  size at most $r$.  

We believe that $r$-ample finite simplicial complexes, with $r$ large, can serve as models for very resilient networks in real life applications. 

One of the results of \cite{farber2019rado} states that the geometric realisation of the Rado complex is homeomorphic to the infinite dimensional simplex and hence it is contractible. A related mathematical object is {\it the medial regime random simplicial complex} \cite{farber2020random} which, as we show in this paper, is $r$-ample, with probability tending to one. 

It was proven in  \cite{farber2020random} that the medial regime random simplicial complex is simply connected and 
has vanishing Betti numbers in dimensions $\le \ln \ln n.$ For these reasons one expects that any $r$-ample simplicial complexes is highly connected, for large $r$. This question is discussed below in \S \ref{4}. 
 
Analogues of the ampleness property of this paper have been studied for graphs, hypergraphs, tournaments, and other structures, in combinatorics and in mathematical logic.
In the literature a variety of terms have been used: 
\emph{$r$-existentially completeness}, \emph{$r$-existentially closedness}, \emph{$r$-e.c.}~for short~\cite{cherlin1993combinatorial, bonato2009search}, and also the \emph{Adjacency Axiom $r$} \cite{blass1979properties, blass1981paley},  an \emph{extension property} \cite{fagin1976probabilities},  \emph{property $P(r)$} \cite{bollobas1985random, caccetta1985property}. 
This property intuitively means that you can get anything you want, it is also referred to as the \emph{Alice's Restaurant Axiom} \cite{winkler1993random, spencer1993zero}, and sometimes just as \emph{random}. 
Here we use the term \emph{$r$-ample}. 

The plan of this paper is as follows. In \S \ref{defs} we give the main definition and discuss several examples. 

In \S \ref{3} we discuss the resilience properties of $r$-ample complexes; our main result, Theorem \ref{remove1}, gives a bound on 
the family of simplexes $\mathcal F$ such that removing $\mathcal F$ reduces the level of ampleness by at most $k$. 
A significant role in this estimate plays the Dedekind number which equals the number of simplicial complexes on $k$ vertexes; good asymptotic approximations for the Dedekind number are known, see \S \ref{defs}.

In the following section \S \ref{4} we show that $r$-ample simplicial complexes are simply connected and 2-connected, for suitable values of $r$. Note that the Rado complex is contractible and hence one expects that any $r$-ample complex is $k$-connected for $r>r(k)$, for some $r(k)<\infty$. We do not know if this is true in general, however we are able to analyse the cases $k=1$ and $k=2$. 

In \S \ref{5} we construct ample simplicial complexes using the probabilistic method. We show that 
for every $r \geq 5$ and for any $n \ge  r2^r2^{2^{r}}$, there exists an $r$-ample simplicial complex having exactly $n$ vertexes, see Proposition \ref{proba}. 


In the final section 6 we construct an explicit family, in the spirit of Paley graphs, of r-ample simplicial complexes on 
$\exp(O(r2^r))$ vertices. We do not know if this construction gives smaller r-ample complexes than in our probabilitic existence proof.
The lower bound $\exp(\Omega(2^r/\sqrt{r}))$ follows from Corollary \ref{lower}. 
The construction of this section uses sophisticated tools of number theory. 

\section{Definitions and first examples}
\label{defs}

First we fix our notations and terminology. The symbol $V(X)$ denotes the set of vertices of a simplicial complex $X$. 
If $U\subseteq V(X)$ is a subset we denote by $X_U$ the {\it induced subcomplex} on $U$, i.e., $V(X_U)=U$ and a set of points of $U$ forms a simplex in $X_U$ if and only if it is a simplex in $X$. 

An {\it embedding} of a simplicial complex $A$ into $X$ is an isomorphism between $A$ and an induced subcomplex of $X$.

 The \emph{join} of simplicial complexes $X$ and $Y$ is denoted $X \ast Y$; recall that the set of vertexes of the join is $V(X)\sqcup V(Y)$ and the simplexes of the join are simplexes of the complexes $X$ and $Y$ as well as simplexes of the form
 $\sigma\tau=\sigma\ast\tau$ where $\sigma$ is a simplex in $X$ and $\tau$ is a simplex in $Y$. 
 The simplex $\sigma\tau=\sigma\ast\tau$ has as its vertex set the union of vertexes of $\sigma$ and of $\tau$. 
 The symbol $CX$ stands for the \emph{cone} over $X$. For a vertex $v\in V(X)$ the symbol $\lk_X(v)$ denotes the {\it link} of $v$ in $X$, i.e., the subcomplex of $X$ formed by all simplexes $\sigma$ which do not contain $v$ but such that $\sigma v=\sigma\ast v$ is a simplex of $X$. 

Here is our main definition.

\begin{definition}
\label{first}
\label{def:ample}
Let $r\ge 1$ be an integer. A nonempty simplicial complex $X$ is said to be {\it $r$-ample} if for each subset $U\subseteq V(X)$ with $|U|\le r$ and for each subcomplex $A\subseteq X_U$ there exists a vertex $v\in V(X) - U$ such that 
\begin{eqnarray}\label{11a}
{\rm lk}_X(v)\cap X_U \;=\; A.
\end{eqnarray} 
We say that $X$ is {\it ample or $\infty$-ample} if it is $r$-ample for every $r\ge 1$. 
\end{definition}

The condition (\ref{11a}) 
can equivalently be expressed as 
\begin{eqnarray}\label{12}
X_{U\cup \{v\}} \;=\; X_U \cup (v\ast A). 
\end{eqnarray}

Obviously, no finite simplicial complex can be $\infty$-ample. 
In ~\cite{farber2019rado} it was shown that {\it there exists a unique, up to isomorphism, $\infty$-ample simplicial complex $X$ on countably many vertices}.
The 1-skeleton of $X$ is the well known Rado graph
~\cite{erdHos1963asymmetric, rado1964universal, cameron1997random}. 
The complex $X$ (called {\it the Rado complex} in \cite{farber2019rado}) possesses remarkable stability properties, for example removing from $X$ any finite set of simplexes leaves a simplicial complex isomorphic to $X$.

Later in this paper we show that  for any integer $r$ there exist finite $r$-ample simplicial complexes and 
we give probabilistic as well as explicit deterministic constructions. 

It is clear that $r$-ampleness depends only on the $r$-dimensional skeleton. 

To be 1-ample a simplicial complex must have no isolated vertexes and no vertexes connected to all other vertexes. 

A 1-ample complex has at least 4 vertexes and Figure \ref{fig:1ample} shows two such examples. 
\begin{figure}[h]
\centering
\includegraphics[scale = 0.4]{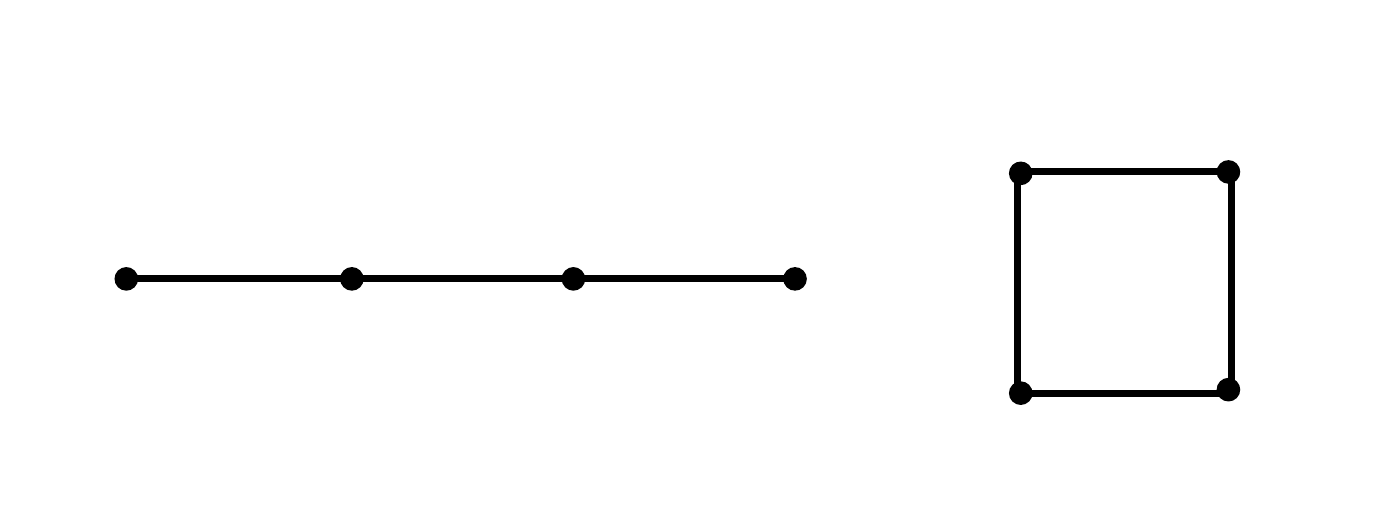}
\caption{1-ample complexes.}
\label{fig:1ample}
\end{figure}

A 2-ample complex is connected since for any pair of vertexes there must exist a vertex connected to both, i.e. the complex must have diameter $\le 2$. 
A 2-ample complex is also {\it twin-free} in the sense that no two vertexes have exactly the same link. 
The following example shows that a 2-ample simplicial complex may be not simply connected. 

\begin{example}
\label{thirteen} 
Consider a 2-dimensional simplicial complex $X$ having 13 vertexes labelled by integers $0, 1, 2, \dots, 12$. 
A pair of vertexes $i$ and $j$ is connected by an edge iff the difference $i-j$ is a square modulo 13, i.e. if
$$i-j \equiv \pm 1, \pm 3, \pm 4 \mod 13.$$ The 1-skeleton of $X$ is a well-known Paley graph of order 13. 
Next we add 13 triangles $$i, i+1, i+4, \quad \mbox{ where} \quad i=0, 1, \dots, 12.$$
We claim that the obtained complex $X$ is 2-ample. 
The verification amounts to the following: for any two vertices, there exists other ones adjacent to both, neither, only one, and only the other. 
Moreover, any edge lies both on a single filled and unfilled triangles. Indeed, an edge $i, i+1$ lies in the triangle $i, i+1, i+4$ (filled) as well as in the triangle $i-3, i, i+1$ (unfilled). Informally, the filled triangles can be characterised by the identity $1+3=4$ and the unfilled by $3+1=4$. 

We note that $X$ can be obtained from the triangulated torus with 13 vertexes, 39 edges and 26 triangles (see Figure
\ref{torus}) by removing 13 white triangles of type 
$i, i+3, i+4$. 
From this description it is obvious that $X$ collapses onto a graph and calculating the Euler characteristic we find $b_0(X)=1,$ $b_1(X)= 14$ and $b_2(X)=0$. 

\begin{figure}
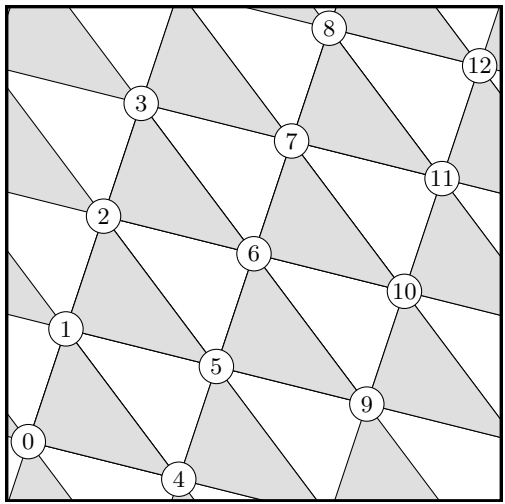
\label{torus}
\centering
\tikz[scale=0.5]{
\tikzstyle{vertex}=[circle, draw, minimum size=13,inner sep=0,fill=white, text=black];
\clip (-0.6,-0.6) rectangle (12.6,12.6);
\foreach[evaluate={\y={Mod(3*\x+1,13)}}] \x in {0,...,12}
\filldraw[fill=gray!25] (\x,\y) -- (\x-3,\y+4) -- (\x-4,\y+1) -- cycle (\x,\y) -- (\x+4,\y-1) -- (\x+1,\y+3) -- cycle (\x,\y) -- (\x-1,\y-3) -- (\x+3,\y-4) -- cycle;
\foreach[evaluate={\y={Mod(3*\x+1,13)}}] \x in {0,...,12} \node[vertex] (\x) at (\x,\y) {\footnotesize\x};
\draw[line width=2] (-0.6,-0.6) rectangle (12.6,12.6);
}
\caption{
 The simplicial complex of Example~\ref{thirteen} can be obtained from the triangulated torus with 13 vertices, 39 edges and 26 triangles, by removing 13 triangles of type $\{i, i+3, i+4\}$. 
 }
\end{figure}

\end{example}

The following Lemma gives an equivalent criterion for $r$-ampleness. 

\begin{lemma}\label{lm:equiv}
\label{AB}
A simplicial complex $X$ is $r$-ample if and only if for every pair $(A,B)$ consisting of a simplicial complex $A$ and an induced subcomplex $B$ of ~$A$, satisfying $|V(A)|\le r+1$, and for every embedding $f_B$ of $B$ into $X$, there exists an embedding $f_A$ of $A$ into $X$ extending $f_B$. 
\end{lemma}

\begin{proof} Clearly the property described in Lemma \ref{AB} implies $r$-ampleness and we only need to show the inverse. Suppose that $X$ is $r$-ample and let $(A,B)$ be a pair consisting of a simplicial complex $A$ with $|V(A)|\le r+1$ and its induced subcomplex $B$. We can find a chain of subcomplexes $$B=B_0\subset B_1\subset B_2\subset \dots\subset B_k=A$$
where each subcomplex $B_{i+1}$ is obtained from $B_i$ by adding a vertex $v_{i+1}$ and attaching a cone $v_{i+1}\ast Y_i$ where $Y_i\subset B_i$ is a subcomplex. Here $V(B_i)\le r$ for any $i$. Once $B=B_0$ is identified with an induced subcomplex of $X$ we may apply inductively the definition
to extend this embedding to an embedding of $A$. 
\end{proof}

Applying Lemma \ref{AB} in the case when $B$ is a single vertex, we obtain:

\begin{corollary} \label{lemma:any}
If $X$ is $r$-ample then any simplicial complex on at most $r+1$ vertexes can be embedded into $X$. 
  \end{corollary}
\begin{corollary}
The dimension of an $r$-ample simplicial complex $X$ is at least~$r$.
\end{corollary}

We shall denote by $M'(n)$ the number of simplicial complexes with vertexes from the set $\{1, 2, \dots, n\}$. 
The number $M'(n)+1=M(n)$ is known as {\it the Dedekind number}, see \cite{kleitman1975}, it equals the number of monotone Boolean functions of $n$ variables and has some other combinatorial interpretations, for example, it equals the number of antichains in the set of $n$ elements. A few first values of {\it \lq\lq the reduced Dedekind number\rq\rq} $M'(n)$ are $M'(1)=2$, $M'(2)=5$, $M'(3)=19$. For  general $n$,  
$M'(n)$ admits estimates
\begin{eqnarray}\label{dede}
\hskip 1cm
{\binom n {\lfloor n/2 \rfloor}}\,  \le\,  \log_2 (M'(n))\, \le \, {\binom n {\lfloor n/2 \rfloor} \left(1+O\left(\frac{\log n}{n}\right)\right)}.
\end{eqnarray}
The lower bound in (\ref{dede}) is easy: one counts only the simplicial complexes having the full $\lfloor n/2 \rfloor$ skeleton; the upper bound in (\ref{dede}) has been obtained in \cite{kleitman1975}. We shall also mention that 
\begin{eqnarray}
{\binom n {\lfloor n/2 \rfloor}} \sim \sqrt{\frac{2}{\pi n}}\cdot 2^n,
\end{eqnarray}
as follows from the Stirling formula. Thus, 
\begin{eqnarray}\label{frml25}
\log_2 \log_2 (M'(n)) = n - \frac{1}{2} \log_2 n +O(1).
\end{eqnarray}

\begin{corollary}\label{cor26}\label{lower}
An $r$-ample simplicial complex contains at least $$M'(r)+r \ge 2^{\binom r {\lfloor r/2\rfloor}} +r$$
vertexes. 
\end{corollary}
\begin{proof} Let $X$ be an $r$-ample complex.
Using Lemma \ref{lemma:any} we can embed into $X$ an $(r-1)$-dimensional simplex $\Delta$ having $r$ vertexes. Applying Definition \ref{first},
for every subcomplex $A$ of $\Delta$ we can find a vertex $v_A$ in the complement of $\Delta$ having $A$ as its link intersected with $\Delta$. The number of subcomplexes $A$ is $M'(r)$ and we also have $r$ vertexes of $\Delta$ which gives the estimate. 
\end{proof}

The following Lemma gives information about local structure of $r$-ample complexes.

\begin{lemma}\label{cone}
Let $X$ be an $r$-ample simplicial complex. Then any simplicial map $f:K\to X$, where $|V(K)|\le r$, is null-homotopic. 
\end{lemma}
\begin{proof}
The set $U=f(V(k))\subset V(X)$ has cardinality $\le r$ and applying Definition \ref{first} we can find a vertex $v\in V(X)-U$ such that $X_{U\cup \{v\}} =v\ast X_U$ (cone over $X_U$). Thus we see that $f:K\to X$ factorises through a map with values in the cone $v\ast X_U$ which is contractible and hence $f$ is null-homotopic. 
\end{proof}

\section{ Resilience of ample complexes}
\label{3}

In this section we present a few results characterising \lq\lq   resilience\rq\rq\, of $r$-ample simplicial complexes: small perturbations to the complex
reduce its ampleness in a controlled way and hence many important geometric properties pertain. 

The perturbations that we have in mind are as follows. If $X$ is a simplicial complex and $\mathcal F$ is a finite set of simplexes of $X$, one may consider the simplicial complex $Y$ obtained from $X$ by removing all simplexes of $\mathcal F$ as well as all simplexes which have faces 
belonging to $\mathcal F$. We shall say that $Y$ {\it is obtained from $X$ by removing the set of simplexes $\mathcal F$}. 

Below we assume that $\mathcal F$ is {\it \lq\lq small\rq\rq } and $X$ is {\it \lq\lq large\rq\rq}; we are interested in situations when $Y$ preserves certain properties of $X$ despite \lq\lq the damage\rq\rq \ caused by removing 
the family of simplexes $\mathcal F$. 

We shall characterise the size of $\mathcal F$ by two numbers: $|\mathcal F|$ (the cardinality of $\mathcal F$) and 
$ \dim(\mathcal F)=\sum_{\sigma\in \mathcal F} \dim \sigma$ ({\it \lq\lq the total dimension\rq\rq\ of $\mathcal F$}). 
\begin{theorem}\label{remove1}
Let $X$ be an $r$-ample simplicial complex and let $Y$ be obtained from $X$ by removing a set  $\mathcal F$ of simplexes. Then $Y$ is $(r-k)$-ample provided that 
 \begin{eqnarray}\label{less}
  |\mathcal F|+\dim(\mathcal F) <  M'(k)+k.
 \end{eqnarray}
 In particular, taking into account (\ref{dede}), the complex $Y$ is $(r-k)$-ample if 
  \begin{eqnarray}\label{less1}
|\mathcal F|+\dim(\mathcal F) < 2^{\binom k {\lfloor k/2\rfloor}}+k . 
 \end{eqnarray}
 \end{theorem}
\begin{proof} Without loss of generality we may assume that $\mathcal F$ forms an anti-chain, i.e. 
no simplex of $\mathcal F$ is a proper face of another simplex of $\mathcal F$. 
Indeed, if $\sigma_1\subset \sigma_2$, where $\sigma_1, \sigma_2\in \mathcal F$,
we can remove $\sigma_2$ from $\mathcal F$ without affecting the complex $Y$.


Consider a vertex $v\in V(Y)$ and its links $\lk_Y(v)\subset \lk_X(v)$ in $Y$ and in $X$, correspondingly. 
Denote by $\mathcal F_v$ the set of simplexes $\sigma\subset \lk_X(v)$ such that either $\sigma\in \mathcal F$ or $v\sigma\in \mathcal F$. As follows directly from the definitions,  {\it $\lk_Y(v)$ is obtained from $\lk_X(v)$ by removing the set of simplexes $\mathcal F_v$}.

Represent $\mathcal F$ as the disjoint union 
$$\mathcal F \, =\, \mathcal F_0\sqcup \mathcal F_{1},$$
 where $\mathcal F_0$ is the set of zero-dimensional simplexes from $\mathcal F$ and $\mathcal F_{1}$ is the set of simplexes in $\mathcal F$ having dimension $\ge 1$. 
 Denote by 
 $$W_0=\cup_{\sigma\in \mathcal F_0} V(\sigma) \quad \mbox{and}\quad W_1=\cup_{\sigma\in \mathcal F_1}V(\sigma)$$ 
 the sets of zero-dimensional simplexes and the set of vertexes of simplexes of positive dimension in $\mathcal F$. 
 Note that $W_0\cap W_1=\emptyset$ due to our anti-chain assumption. Besides, $V(Y)=V(X)-W_0$ and therefore 
 $W_1\subset V(Y)$. 

Let  $U\subset V(Y)$ be a subset and let $v\in V(Y)$ be a vertex such that: \newline (a) $v\notin W_1$ and (b) 
the set $\lk_X(v)\cap X_U$ is a subcomplex of $Y_U$. Then 
\begin{eqnarray}\label{conclusion}
\lk_Y(v)\cap Y_U\, =\, \lk_X(v)\cap X_U. 
\end{eqnarray}
Indeed, we have $\lk_X(v)\cap Y_U =\lk_Y(v)\cap Y_U$ because of our assumption (a) and  
$\lk_X(v)\cap Y_U = \lk_X(v)\cap X_U$ because of (b).

Let $k$ be an integer satisfying (\ref{less}) and 
let $U\subset V(Y)$ be a subset with $|U|\le r-k$. 
Given a subcomplex $A\subset Y_{U}$, we want to show the existence of a vertex $v\in V(Y)-U$ such that 
\begin{eqnarray}\label{inters}
\lk_Y(v)\cap Y_{U} =A.\end{eqnarray}
This would mean that our complex $Y$ is $(r-k)$-ample.

Consider the induced subcomplex $X_{U}$ which obviously contains $A$ as a subcomplex. Consider also the abstract simplicial complex 
$$K=X_{U}\cup (A\ast \Delta),$$ 
where $\Delta$ is an abstract full simplex on $k$ vertexes. 
Note that 
$K$ has at most $r$ vertexes, 
$X_{U}$ is an induced subcomplex of $K$ and it is naturally embedded into $X$. 
Using the assumption that $X$ is $r$-ample and applying Lemma \ref{AB}, we can find an embedding of $K$ into $X$ extending the identity map of $X_{U}$. 
In other words, we can find $k$ vertexes $v_1, \dots, v_k\in V(X)-U$ such that for a simplex $\tau$ of $X_{U}$ 
and for any subset $$\tau'\subset \{v_1, \dots, v_k\}=U'$$ one has 
$\tau\tau' \in X$ if and only if $\tau \in A$. 
If one of these vertexes $v_i$ lies in $V(Y)-W_1$ then (using (\ref{conclusion}))
$$\lk_Y(v_i)\cap Y_{U}=\lk_X(v_i)\cap X_{U}=A$$ 
and we are done. Thus, without loss of generality, we can assume that 
\begin{eqnarray}\label{includ}
U'\subset W_0\cup W_1.
\end{eqnarray}

Let $Z\subset \Delta$ be an arbitrary simplicial subcomplex. 
We may use the $r$-ampleness of $X$ and apply Definition \ref{first} to the subcomplex $A\sqcup Z$ of 
$X_{U\cup U'}$. 
This gives a vertex $v_Z\in V(X)-(U\cup U')$ satisfying 
$$\lk_{X}(v_Z) \cap X_{U\cup U'}  =A\sqcup Z$$ and in particular, 
\begin{eqnarray}\label{inters2}
\lk_{X}(v_Z) \cap X_{U}  =A.\end{eqnarray}
For distinct subcomplexes $Z, Z'\subset \Delta$ the points $v_Z$ and $v_{Z'}$ are distinct and the cardinality of the set 
$\{v_Z; Z\subset \Delta\}$ equals $M'(k)$. 
Noting that (\ref{inters2}) is a subcomplex of $Y_U\subset X_U$ and 
comparing (\ref{conclusion}), (\ref{inters}), (\ref{inters2}), we see that our statement would follow once we know that 
the vertex $v_Z$ lies in $V(Y) -W_1$ at least for one subcomplex $Z$. 

Let us assume the contrary, i.e. 
$v_Z \in (W_0\cup W_1) - U'$
for every subcomplex $Z\subset \Delta$. The cardinality of the set $\{v_Z\}$ equals $M'(k)$ and the cardinality of the set $(W_0\cup W_1) - U'$ equals $|\mathcal F|+\dim \mathcal F - k$ and we get a contradiction with our assumption 
(\ref{less}). 

This completes the proof. 
\end{proof}

\begin{cor}\label{vertexcon}
Let $X$ be an $r$-ample simplicial complex and let $Y$ be obtained from $X$ by removing a set $\mathcal F$ of simplexes. Denote by $a_i$ the number of $i$-dimensional simplexes in $\mathcal F$ where $i=0, 1, \dots.$. Then:
\begin{enumerate}
\item[(a)] If $r\ge 3$ and  $a_0+2a_1 < M'(r-2) +r-2$ then $Y$ is path-connected. 
In particular, $Y$ is path-connected if
 $$a_0+2a_1 < 2^{\binom {r-2}{\lfloor r/2\rfloor -1}} + r-2.$$
 \item[(b)] If $r\ge 5$ and $a_0+2a_1 +3a_2 < M'(r-4) +r-4$ then $Y$ is simply connected. In particular, $Y$ is simply connected if 
 $$a_0+2a_1 +3a_2 < 2^{\binom {r-4}{\lfloor r/2\rfloor -2}} + r-4.$$
 
 \item[(c)] If $r\ge 19$ and $a_0+2a_1 +3a_2+4a_3< M'(r-18) +r-18$ then $Y$ is 2-connected. In particular, $Y$ is 2-connected if 
 $$a_0+2a_1 +3a_2+4a_3 < 2^{\binom {r-18}{\lfloor r/2\rfloor -9}} + r-18.$$
 \end{enumerate}
\end{cor}
\begin{proof}
Claim (a) follows from Theorem \ref{remove1} and from the observation that a 2-ample complex is connected; claim (b) follows from Theorem \ref{remove1} and Proposition \ref{lem1cycle}; claim (c) follows from Theorems \ref{remove1} and  \ref{2conn}. 
\end{proof}

We finish this section with the following observation. 
\begin{prop}
\label{lem link}
The link of a vertex in an $r$-ample simplicial complex is $(r-1)$-ample. More generally, 
 the link of every k-dimensional simplex in an $r$-ample complex is $(r-k-1)$-ample. 
\end{prop}

\begin{proof} We consider the case $k=0$ first. 
Let $v\in V(X)$ be a vertex and let $L$ denote the link of $v$ in $X$. Let $(A, B)$ be a pair consisting of a simplicial complex $A$ and its induced subcomplex $B$ where $|V(A)|\le r$. Consider the pair $(CA, CB)$ consisting of cones with apex~$w$. Note that $CB$ is an induced subcomplex of $CA$ and $|V(CA)|\le r+1$. Since $v \ast L \subseteq X$, any embedding $f_B:B\to L$ can be extended to an embedding $f_{CB}: CB\to X$ where $w$ is mapped into $v$. Since $X$ is $r$-ample, applying Lemma~\ref{lm:equiv} we can find an embedding $f_{CA}: CA\to X$ extending $f_{CB}.$ Then the restriction $f_{CA}|A$ is an embedding $A\to L$ extending $f_B$. 

The case when $k>0$ is similar. Let $\sigma$ be a $k$-simplex in $X$ and let $L$ denote its link. 
Consider a pair $(A, B)$ with $|V(A)|\le r-k$, an induced subcomplex $B$ of $A$ and an embedding $f_B: B\to L$. Consider the joins $A'=A\ast \sigma$ and $B'=B\ast \sigma$ and note that $V(A') \le r+1$ and $B'$ is an induced subcomplex of $A'$. By Lemma \ref{lm:equiv} the join embedding $f_{B'}=f_B\ast 1: B'=B\ast \sigma\to L\ast \sigma$ can be extended to an embedding
$f_{A'}: A' \to L\ast \sigma$ which restricts to an embedding $f_A: A\to L$ extending $f_B$. 
\end{proof}

\section{Higher connectivity of ample complexes}
\label{4}

It is natural to ask whether the geometric realisation of an $r$-ample simplicial complex is highly connected, i.e. has vanishing homotopy groups below certain dimension. The motivation for this question comes from the fact that 
an $r$-ample finite simplicial complex can be viewed as an approximation to the Rado simplicial complex whose geometric realisation is homeomorphic to an infinite dimensional simplex and hence is contractible \cite{farber2019rado}. 

Recall that a simplicial complex $Y$ is $m${\it -connected} if for every triangulation of the $i$-dimensional sphere $S^i$ 
with $i\le m$ and for every 
simplicial map $\alpha: S^i\to Y$ there exists a triangulation of the disc $D^{i+1}$ extending the given triangulation of the sphere $S^i=\partial D^{i+1}$ and a simplicial map 
$\beta: D^{i+1}\to Y$ extending $\alpha$. A 1-connected complex is also said to be simply connected. 

\begin{prop}\label{lem1cycle}
For $r\ge 4$, any $r$-ample simplicial complex $Y$ is simply connected. 
Moreover, any simplical loop $\alpha: S^1\to Y$ 
with $n$ vertexes in an $r$-ample complex $Y$ bounds a simplicial disc $\beta: D^2\to Y$ where $D^2$ is a triangulation of the disc having $n$ boundary vertexes, at most  
$\lceil\frac{n-3}{r-3}\rceil $ internal vertexes and at most $\lceil \frac{n-3}{r-3}\rceil \cdot(r-1)+1$ triangles. 
\end{prop}
\begin{proof}
If $n\le r$ we may simply apply the definition of $r$-ampleness and find an extension $\beta: D^2\to Y$ with a single internal vertex. If $n>r$ we may apply the definition of $r$-ampleness to any arc consisting of $r$ vertexes, see Figure
\ref{fig:gamma}. This reduces the length of the loop by $r-3$ and performing $\lceil\frac{n-r}{r-3}\rceil$ such operations we obtain a loop of length $\le r$ which can be filled by a single vertex. The number of internal vertexes of the bounding disc will be $\lceil\frac{n-r}{r-3}\rceil +1= \lceil\frac{n-3}{r-3}\rceil.$
\begin{figure}[h]
    \centering
    \includegraphics[scale = 0.6]{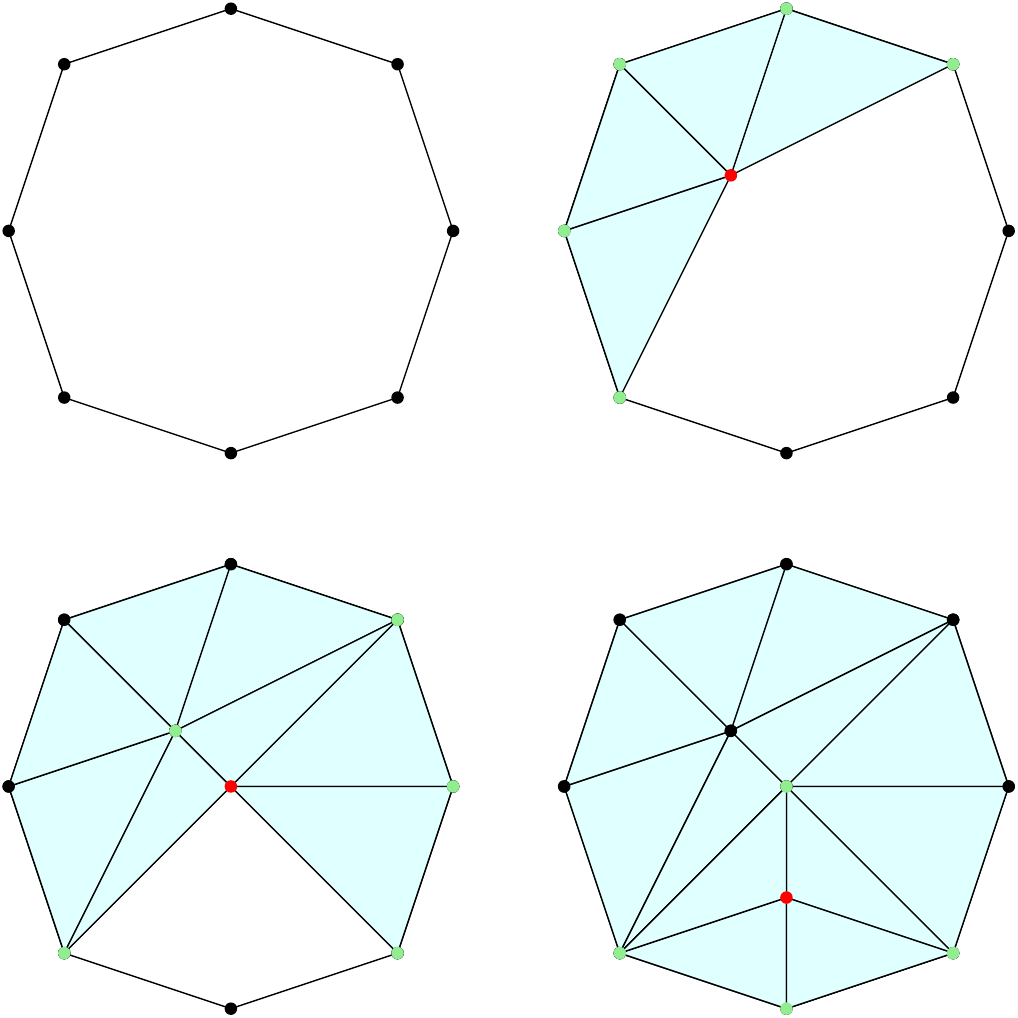}
    \caption{The process of constructing the bounding disc in a $5$-ample complex as detailed in the proof of Proposition~\ref{lem1cycle}}
    \label{fig:gamma}
\end{figure}
To estimate the number of triangles we note that on each intermediate step of the process described above we add $r-1$ triangles and on the final step we may add at most $r$ triangles. This leads to the upper bound $\lceil \frac{n-r}{r-3}\rceil \cdot (r-1) +r= 
\lceil \frac{n-3}{r-3}\rceil \cdot(r-1)+1$. 
\end{proof}

Currently we have no examples of $3$-ample complexes which are not simply connected. 
The 2-ample complex of Example \ref{thirteen} is not simply connected. 

Next we state our result on 2-connectivity of $r$-ample complexes:

\begin{theorem}\label{2conn}
For $r\ge 18$, every $r$-ample simplicial complex is 2-connected. 
\end{theorem}

In the proof (which is given below) we shall use the following property of triangulations of the 2-sphere. 
By the degree $d_v$ of a vertex $v$ of a triangulation $\Sigma$ of $S^2$ 
we understand the number of edges incident to $v$.
\begin{lemma}\label{11}
In any triangulation $\Sigma$ of the 2-dimensional sphere there exist two adjacent vertexes $v$ and $w$ satisfying
$d_v\le 11$ and $d_w\le 11$. 
\end{lemma}
\begin{proof}[Proof of Lemma \ref{11}] Recall that for any triangulation $\Sigma$ of the 2-sphere one has the following relation
\begin{eqnarray}\label{GB}
\sum_v \left(1-\frac{d_v}{6}\right) =2,
\end{eqnarray}
where $v$ runs over all vertexes of $\Sigma$ and $d_v$ denotes the degree of the vertex $v$. Formula (\ref{GB}) is well-known, it follows from the Euler's formula $V-E+F=2$ by observing that 
$E=\frac{1}{2} \sum_v d_v$ and $F= \frac{1}{3} \sum_v d_v$. Formula (\ref{GB}) can be viewed as a combinatorial version of the Gauss-Bonnet theorem. 

Let $A$ stands for the set of vertexes $v\in V(\Sigma)$ satisfying $d_v\le 11$ and let $B$ denote the complementary set consisting of vertexes with $d_v\ge 12$. Denote also 
$$C_A=\sum_{v\in A} \left(1-\frac{d_v}{6}\right), \qquad C_B=\sum_{v\in B} \left(1-\frac{d_v}{6}\right),$$
the contributions of both sets into the sum (\ref{GB}). 
Since $d_v\ge 3$ we have $1-\frac{d_v}{6}\le \frac{1}{2}$ and hence
$$C_A \le \frac{1}{2} |A|.$$  Besides, $1-\frac{d_v}{6}\le -1$ for $v\in B$ and therefore
$$ C_B \le -|B|, \qquad C_A +C_B =2, \qquad |A|+|B|=V.$$
From these relations one obtains 
\begin{eqnarray}\label{geq}
|A|\ge \frac{2}{3}(V+2).
\end{eqnarray}
Next we claim that there must exist an edge $e$ with both endpoints in $A$, i.e. having degree $\le 11$. Assuming the contrary, every triangle of the triangulation $\Sigma$ would have at most one vertex of degree $\le 11$ and since the minimal degree is $3$, using (\ref{geq}), we obtain that the number of triangles would be at least $$3\cdot \frac{2}{3}(V+2)=2V+4.$$ However this contradicts the well-known relation $F=2V-4$ for the total number of triangles. 
\end{proof}

We shall also need the following simple Lemma:

\begin{lemma}\label{connect}
Let $D$ be a triangulated $2$-dimensional disk and let $L=\partial D$ be its boundary circle. Assume that the length  
(i.e. the number of edges) of $L$ is at least $7$ and $D$ has at most one internal vertex. 
Then there exists a pair of boundary vertexes $x, y\in L$ satisfying $d_L(x, y)\ge 3$ and such that they 
can be connected by a simplicial simple arc $\alpha$ in $D$ with 
$\partial \alpha = \{x, y\}= \alpha\cap \partial D$.
Here $d_L(x, y)$ denotes the distance between $x$ and $y$ along the boundary $L$, i.e. the number of edges in the shortest simplicial path in $L$ connecting $x$ and $y$. 
\end{lemma}

\begin{proof} Let us first consider the case when $D$ has no internal vertexes. 
Denoting the length $|L|$ of the boundary by $n$, we see that there are $n-3$ internal arcs (as follows from the Euler's formula). We want to show that there exists an internal arc such that its end points $x, y$ satisfy 
$d_L(x, y)\ge 3$. Assuming that $d_L(x, y)= 2$ for any internal arc, we may
cut $D$ along an arbitrary internal arc which produces a triangle and a triangulated disk $D'$ with $|L'|=n-1$ where 
$L'=\partial D'$. If we knew that our statement was true for $D'$ we could find vertexes $x, y\in L'$ satisfying 
$d_{L'}(x, y)\ge 3$ such that $x, y$ are the endpoints of an internal arc of $D'$. Then $d_L(x,y)\ge d_{L'}(x, y) \ge 3$. 
This argument shows that without loss of generality we may assume that the length of $L$ is exactly $7$ but in this case 
one can see that our statement holds by examining a few explicit cases; see left part of Figure \ref{fig:triangulation}. 

Consider now the case when $D$ has a single internal vertex, denoted $v$. The vertex $v$ is connected to at least $3$ other vertexes $a, b, c\in L$. Let $d_L(a, c)$ be the maximal among the three numbers $d_L(a, b)$, $d_L(a, c)$, 
$d_L(b, c)$. Then either $d_L(a, b) + d_L(b, c) +d_L(a,c) = |L|$ or $d_L(a, c) = d_L(a, b) +d_L(b,c)$. 
In the first case one obtains $d_L(a, c)\ge 4$ (since $|L|\ge 7$) and we are done, as we can take for $\alpha$ the arc $av+vc$. In the second case we may similarly treat the case $d_L(a, b)\ge 3$ and we are left with the possibility 
$d_L(a, b)=2$ and hence $d_L(a, c)=1$ and $d_L(c, b)=1$. Cutting along the arc $av+vc$ produces two triangulated disks, each with no internal vertexes, one having 4 vertexes and the other, denoted $D'$, having $|L|$ vertexes. We see that $|\partial D'|\ge 7$ and hence we may apply the previous case of the Lemma, i.e. we can find two vertexes 
$x, y\in \partial D'=L'$ connected by an internal arc such that $d_{L'}(x, y)\ge 3$. We are done if none of the points 
$x, y$ equal $v$. However if $x=v$ we may consider the pair $y, \ b\in L$ since $d_L(y, b)=d_{L'}(y, v) \ge 3$ and the points 
$y$, $b$ are connected by the arc $\alpha = yv + vb$. See Figure \ref{fig:triangulation}, right.
\end{proof} 
\begin{figure}[h]
    \centering
    \includegraphics[scale = 0.5]{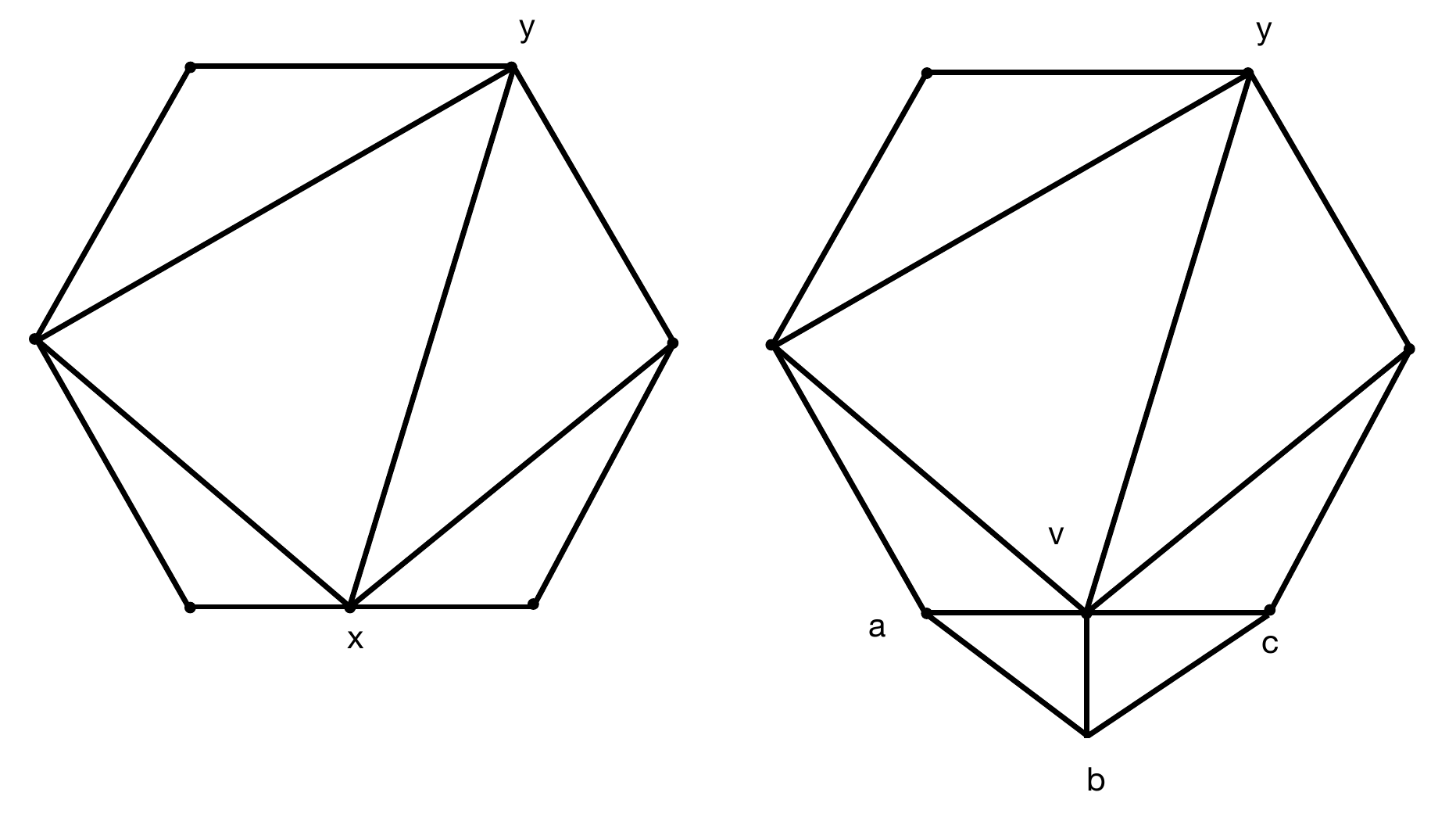}
    \caption{Triangulated disks with no internal vertexes (left) and one internal vertex (right).}
    \label{fig:triangulation}
\end{figure}
\begin{proof}[Proof of Theorem \ref{2conn}] We shall assume the contrary and arrive to a contradiction. Let $Y$ be an $18$-ample simplicial complex which is not 2-connected. From Proposition \ref{lem1cycle} we know that $Y$ is simply connected. 
Let $M(Y)$ denote the smallest number of vertexes in a triangulation $\Sigma$ of the sphere $S^2$ admitting a simplicial 
essential (i.e. not null-homotopic) map $f: \Sigma \to Y$. By the well-known Simplicial Approximation Theorem, 
$M(Y)$ is finite. Lemma \ref{cone} implies $V(\Sigma)\ge 19$ for every simplicial essential map $f:\Sigma\to Y$ and hence 
$M(Y)\ge 19$. 


Let $f:\Sigma\to Y$ be an essential simplicial minimal map, i.e.  $V(\Sigma)=M(Y)$. 
We shall use the following geometric property of the triangulation $\Sigma$ of $S^2$, its {\it \lq\lq roundness\rq\rq},\, which is described below. 
Suppose that $L\subset \Sigma$ is a simple simplicial loop of
length $|L|\le 18$, i.e. $L$ contains at most 18 edges. 
Clearly, $L$ divides the sphere $\Sigma$ into two triangulated disks $D_1$ and $D_2$; each of these disks having 
$|L|$ boundary vertexes and possibly a number of internal vertexes. We claim that {\it at least one of the disks $D_1$,\ $D_2$ has at most one internal vertex}. Indeed, suppose that each of the disks $D_1$ and $D_2$ has $\ge 2$ internal vertexes. Let $D=a\ast L$ be the cone with apex $a$ and base $L$. We can form two triangulated spheres 
$\Sigma_1 = D_1\cup D$ and $\Sigma_2=D_2\cup D$ and each of these spheres has strictly smaller number of vertexes than $\Sigma$ (since $D$ has a single internal vertex and each of the disks $D_1$, $D_2$ has at least 2 internal vertexes). 
Next we observe that each of the spheres $\Sigma_1$ and $\Sigma_2$ can be mapped simplicially into $Y$ such that 
at least one of the maps $\Sigma_1\to Y$ or $\Sigma_2\to Y$ is essential. Indeed, 
consider the image $f(L)\subset Y$ of the loop $L$ in $Y$. It is a subcomplex with at most $18$ vertexes and by 
$18$-ampleness of $Y$ we can find a vertex $u\in V(Y)$ such that $u\ast f(L)\subset Y$. Now we may extend the map $f:\Sigma\to Y$ onto the disk $D=a\ast L$ by mapping $a$ onto $u$ and extending this map onto the cone by linearity. 
We obtain a simplicial map 
$g: \Sigma\cup D\to Y$ extending $f$ and the restrictions $g_1=g|_{\Sigma_1}: \Sigma_1\to Y$ and $g_2=g|_{\Sigma_2}: \Sigma_2\to Y$ are the desired simplicial maps. Since $\Sigma_1\cup \Sigma_2=\Sigma\cup D$ and 
$\Sigma_1\cap \Sigma_2= D$ is contractible, we see that $g$ is essential (as $f=g|_\Sigma$ is essential) and hence 
at least one of the maps $g_1$, $g_2$ is essential.  Thus, we arrive at a contradiction with the minimality of $f$. 

%
%
Next we invoke Lemma \ref{11} which gives us two adjacent vertexes $v$ and $w$ of $\Sigma$, each having degree at most $11$. Let  $e$ be the edge connecting $v$ and $w$. 
Consider the 
subset $U$ of the surface $\Sigma$ which is the union of all triangles incident to $e$. The boundary $\partial U$ is a closed curve (potentially with some identifications, see below) formed by $d_v +d_w-4\le 11+11-4 =18$ edges and the interior of $U$ is the union of $d_v+d_w-2$ triangles. 
\begin{figure}[h]
    \centering
    \includegraphics[scale = 0.4]{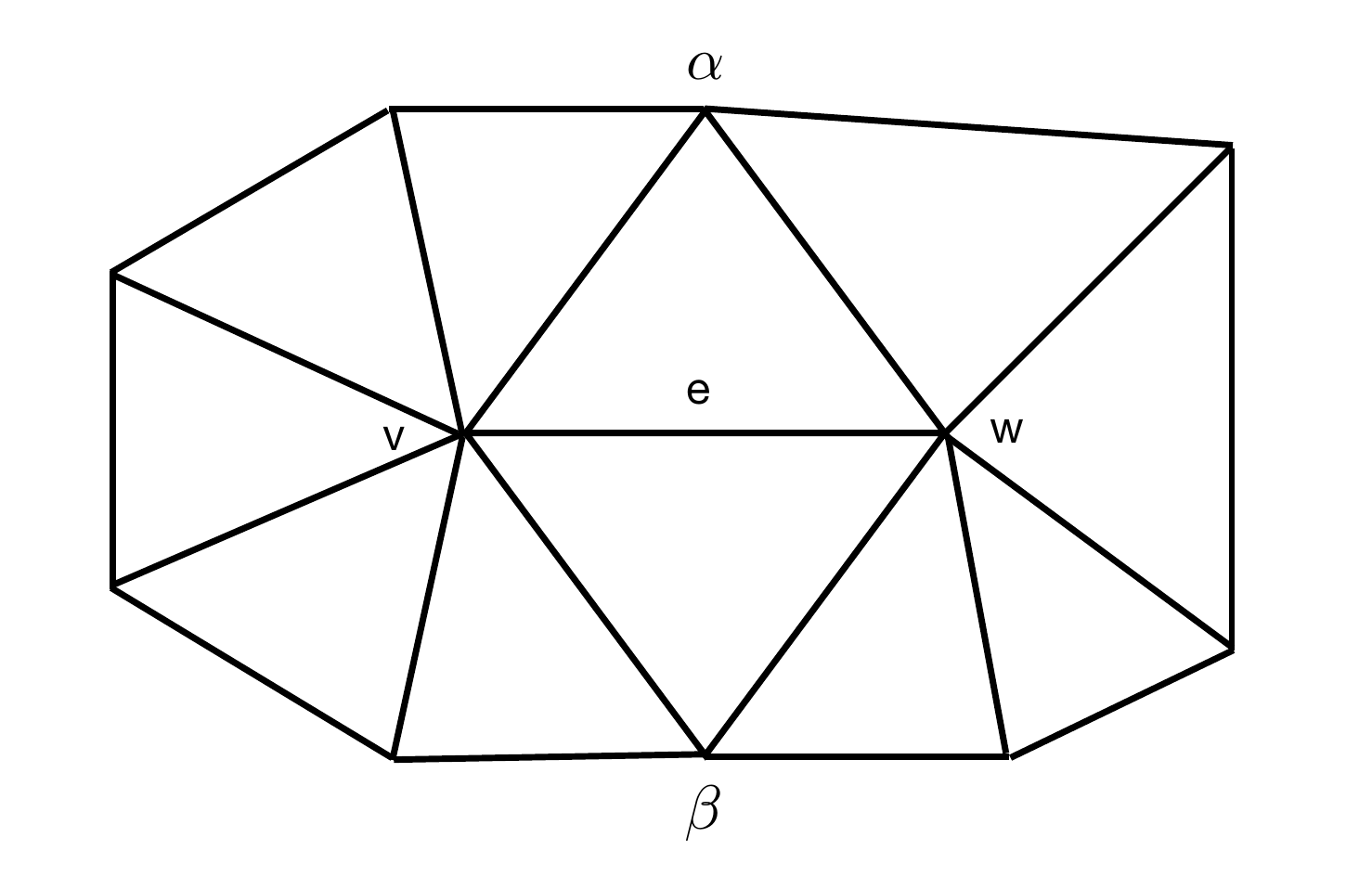}
    \caption{Triangles incident to an edge on the surface.}
    \label{fig:edge}
\end{figure}
The edge $e$ is incident to two triangles; we shall denote by $\alpha$ and $\beta$ the vertexes of these two triangles which are not incident to $e$, see Figure \ref{fig:edge}. 

Let us assume first that the links of the vertexes $v$ and $w$ satisfy $$\lk_\Sigma(v)\cap \lk_\Sigma(w) =\{\alpha, \beta\}.$$
Then $U$ is a triangulated disk with $\le 18+2=20$ vertexes, among them 2 are internal, as shown on Figure \ref{fig:edge}.

\begin{figure}[h]
    \centering
    \includegraphics[scale = 0.45]{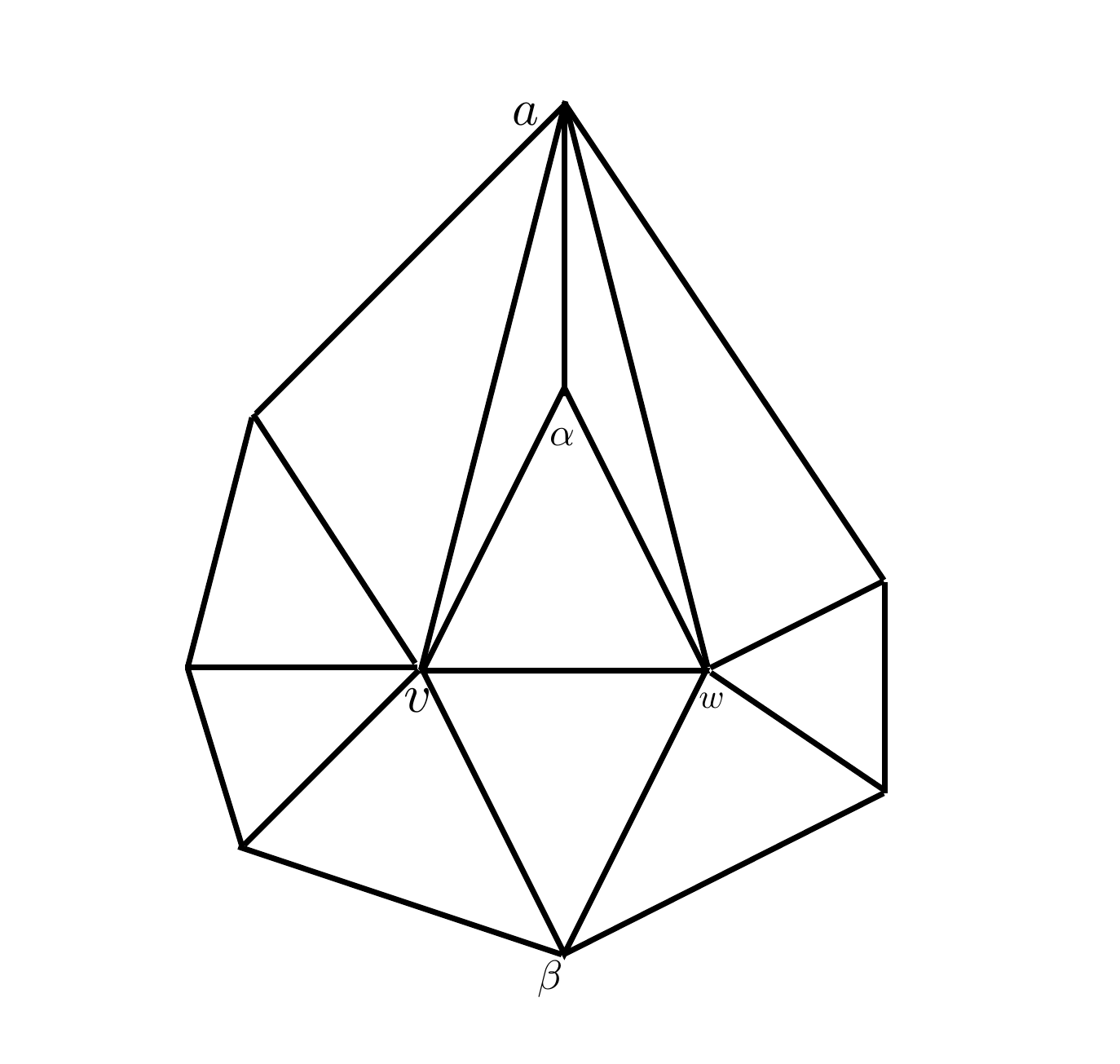}
    \caption{Disk $U$ with 3 internal vertexes.}
    \label{fig:edge2}
\end{figure}

Suppose now that there exists a vertex $a\in\lk_\Sigma (v)\cap \lk_\Sigma(w)$ which is distinct from $\alpha$ and $\beta$. 
Then the
path $L =av+vw+wa$ is a simplicial loop on $\Sigma$ which divides the surface $\Sigma$ into two disks. By the roundness property of $\Sigma$, one of these two disks must have at most one internal vertex. 
In fact, the only possibility is that $L$ 
bounds a disk with one internal vertex and 
 $L$ cannot be the boundary of a triangle: otherwise the edge $e$ would belong to 3 different triangles. 
 It is easy to see that this internal vertex must be either $\alpha$ or $\beta$, 
as there are exactly 2 triangles incident to $e$, 
see Figure \ref{fig:edge2}. In this case $\alpha$ becomes an internal vertex of $U$. 

For similar reasons it might happen that both vertexes $\alpha$ and $\beta$ are internal vertexes of $U$. 

The argument above shows that any vertex lying in $\lk_\Sigma(v)\cap \lk_\Sigma(w)$, which is distinct from $\alpha$ and $\beta$, belongs to a triangular simplicial loop surrounding either 
$\alpha$ or $\beta$ and containing the edge $vw$ (similarly the loop $av+ vw+wa$ shown on Figure \ref{fig:edge2}). This implies that the intersection $\lk_\Sigma(v)\cap \lk_\Sigma(w)$ may contain at most 4 vertexes. 

Potentially it might happen that $U=\Sigma$, i.e. $\partial U=\emptyset$. Then all vertexes of $\Sigma$, other than $v, \ w$, lie in the intersection $\lk_\Sigma(v)\cap \lk_\Sigma(w)$. Using the above arguments, we see that in this case $V(\Sigma)\le 6$, which contradicts our assumption $V(\Sigma)\ge 19$.

The remaining possibility is that $U\subset \Sigma$ is a subcomplex, it has either $2$, $3$ or $4$ internal vertexes and its total number of vertexes is at most $ 20$. 

The closure of the complement of $U$ is in $\Sigma$ another disk, $U'$, and applying the roundness property of 
$\Sigma$, we conclude that that $U'$ has at most one internal point. 
Thus, we see that the triangulation $\Sigma$ must have at most $21$ vertexes in total, and using Corollary \ref{cone} we obtain that $|V(\Sigma)|$ must be equal to one of the three numbers: $19$, $20$ or $21$. 

Using this observation we conclude that the length $\ell= |L|$ of the boundary $L=\partial U=\partial U'$ should satisfy 
$14 \le \ell \le 18.$ 

Finally we show that there must exist a simplicial simple closed curve $L'$ on $\Sigma$ dividing the sphere 
$\Sigma$ into two disks, each having more than one internal points, and this violates the roundness of $\Sigma$ and gives a contradiction. The curve $L'$ is the union of two arcs $L'=A\cup A'$ where $A\subset U$ and $A'\subset U'$. 
We first construct the arc $A'\subset U'$; we only must ensure that (c) the endpoints of $A'$ divide the boundary $L$ into two arcs, 
each of length $\ge 3$. The existence of such an arc follows from Lemma \ref{connect} below. 
Once the arc $A'\subset U'$ satisfying (c) is constructed we connect its endpoints (lying on the boundary $L=\partial U$) by a simple simplicial arc $A$ in $U$;  it is clear from Figures \ref{fig:edge} and \ref{fig:edge2} that any two points on the boundary can be connected by such an arc in $U$.  

The vertexes of $L$ distinct from two vertexes $\partial A=\partial A'$ are internal vertexes of the disks on which 
the sphere $\Sigma$ is divided by the circle $L'$; the condition (c) ensures that at least two vertexes lie in each connected components of $\Sigma - L'$. This contradicts the roundness property of $\Sigma$ and completes the proof of Theorem 
\ref{2conn}
\end{proof}

 We tend to believe that in general, for every $k\ge1$ there exists $r(k)$ such that every $r$-ample simplicial complex is $k$-connected provided that $r\ge r(k)$. We know that $r(1) \le 4$ and $r(2)\le 18$.

\section{Large random simplicial complex is ample}\label{5}

In this section we show that a large random simplicial complex is $r$-ample with probability tending to 1. 
This result implies the existence of $r$-ample finite simplicial complexes, we also use it to estimate the minimal number of vertexes an $r$-ample complex must possess. 

First we recall a model of random simplicial complexes developed in \cite{farber2019lower} and 
\cite{farber2020random}; it 
is a generalisation of the well-known Linial - Meshulam model. 

Let $\Delta_n$ denote the simplex on the vertex set $[n]=\{1, 2, \dots, n\}$. Suppose that for every simplex 
$\sigma\subset \Delta_n$ we are given a probability parameter $p_\sigma\in [0,1]$. We shall use the notation $q_\sigma=1-p_\sigma$. 
For a subcomplex $Y\subset \Delta_n$ we shall denote by $F(Y)$ the set of faces of $Y$ and by $E(Y)$ the set of 
{\it external} faces of $Y$, i.e. simplexes $\sigma\subset \Delta_n$ such that $\partial \sigma\subset Y$ and $\sigma
\not\subset Y$. 
Next we describe a probability measure on the set of all simplicial subcomplexes of $\Delta_n$; in terminology of \cite{farber2019lower} it is the lower measure (as opposed to an upper measure). 
The probability of simplicial subcomplex $Y\subset \Delta_n$ is given by 
\begin{eqnarray}\label{defprob}
P(Y) = \prod_{\sigma\in F(Y)} p_\sigma\cdot \prod_{\sigma\in E(Y)} q_\sigma.
\end{eqnarray}
The intuitive meaning of this probability measure can be described as follows. One constructs a random complex $Y$ inductively, by first selecting a random set of vertexes, adding randomly a set of edges, then a set of 2-simplexes and so on. 
On each step, once the $k$-dimensional skeleton $Y^{k}$ has been constructed, one adds external $(k+1)$-dimensional simplexes of $Y^k$ at random, each with its own probability $p_\sigma$, independently of the other simplexes. We refer to 
 \cite{farber2019lower} for more detail and justification. 
 
 In this section we shall assume that the parameters $p_\sigma$ are in {\it the medial regime}, which means that the probability parameters $p_\sigma$ satisfy
 \begin{eqnarray}\label{med1}
 p_\sigma\in [p, 1-p], 
 \end{eqnarray}
where $p\in (0,1)$ does not depend on $n$. (Note however that in Remark \ref{rem32a} we relax this assumption.) Paper \cite{farber2020random} contains information 
about geometry and topology of medial regime random simplicial complexes; in particular, it is shown that a random simplicial complex in the medial regime is simply connected and has vanishing Betti numbers below dimension
$\sim \log\log n,$ asymptotically almost surely.

\begin{prop}\label{randoma}
For every integer $r\ge 1$, the probability that a medial regime random simplicial complex is $r$-ample tends to one, as $n\to \infty$. 
\end{prop}
\begin{proof}
We estimate probability that a random complex $Y$ is not $r$-ample. 
Let us make the following choices: 
a subset $U\subset [n]$ of cardinality $|U|\le r$, a subcomplex $Z\subset \Delta_n$ with $V(Z)=U$,
 a subcomplex $A\subset Z$  and a vertex $v\in [n]- U$.  Consider the following events
$$W_U = \{Y\subset \Delta_n\mid U\subset V(Y)\},$$ 
$$W_{U,Z}=\{Y\subset \Delta_n\mid \, U\subset V(Y), \, Y_U=Z\},$$
and 
$$W_{U,Z,A, v} = \{Y\subset \Delta_n; \, U\cup \{v\} \subset V(Y), \, Y_{U\cup\{v\}}=Z\cup (A \ast v)\}.$$
Note that $$W_U=\sqcup_Z W_{U,Z}$$ is the disjoint union where $Z$ runs over all subcomplexes of $\Delta_n$ satisfying $V(Z)=U$. Consider also the complement $W_{U,Z,A,v}^c$ of $W_{U,Z,A,v}$ in $W_{U, Z}$, i.e. 
$$W_{U,Z,A,v}^c = W_{U, Z} - W_{U,Z,A,v}.$$
A simplicial complex $Y\subset \Delta_n$ belongs to $W^c_{U,Z,A, v}$ iff $U\subset V(Y)$ and $Y_U=Z$ and either $v\notin V(Y)$ or $v\in V(Y)$ and 
$\lk_Y(v)\cap Y_U\not= A$. We see that for $|U|\le r$ any complex $Y$ lying in the intersection
$$\bigcap_{v\in [n]-U}W^c_{U,Z,A,v}$$
is not $r$-ample. Moreover, the set $\mathcal N$ of all not $r$-ample simplicial complexes $Y\subset \Delta_n$, $Y\not=\emptyset$, coincides with 
$$\mathcal N \, =\, \bigcup_{1\le |U|\le r} \left(\bigsqcup_{Z}\left(\bigcup_{A\subset Z}\left(\bigcap_{v\in [n]-U} W^c_{U,Z,A,v}\right)\right)\right).$$
We denote
$$\mathcal N_{U, Z} = \bigcup_{A\subset Z}\left(\bigcap_{v\in [n]-U} W^c_{U,Z,A,v}\right)$$
and $\mathcal N_{U} =\sqcup_Z \mathcal N_{U, Z}$. Then 
$\mathcal N= \cup_U \mathcal N_{U},$ where $|U|\le r$. 

Using  definition (\ref{defprob}) we can compute conditional probability 
$$P(W_{U,Z,A,v}\, |\, W_{U,Z})= p_v\cdot \prod_{\sigma\in F(A)} p_{v\sigma}\cdot \prod_{\sigma\in E(A|Z)}q_{v\sigma}.
$$
Here $v\sigma$ denotes the join $v\ast \sigma$ and $E(A|Z)$ denotes the set of simplexes $\sigma\in F(Z)-F(A)$ such that
$\partial \sigma\subset A$. Since $|F(A)|+|E(A|Z)|\le 2^r-1$ (note that by definition a simplex is {\it a nonempty} subset of the vertex set), using the medial regime assumption (\ref{med1}), we obtain
\begin{eqnarray}
P(W_{U,Z,A,v}\, \mid\, W_{U,Z}) \ge p^{2^r }.
\end{eqnarray}
Hence the complement $W_{U,Z,A,v}^c$ of $W_{U,Z,A,v}$ in $W_{U, Z}$ satisfies 
$$P(W_{U,Z,A,v}^c\mid W_{U,Z}) \le 1-p^{2^r }$$ and since 
for different vertexes $v\in [n]-U$ the events $W_{U,Z,A,v}^c$ are conditionally independent over 
$W_{U,Z}$ we obtain 
$$P\left(\bigcap_{v} \, W_{U,Z,A,v}^c \mid  W_{U,Z}\right) \le (1-p^{2^r })^{n-|U|}\le (1-p^{2^r })^{n-r}$$ 
and therefore
$$ P(\mathcal N_{U, Z}\mid W_{U,Z}) \, =\, P\left(\bigcup_A \bigcap_{v} \, W_{U,Z,A,v}^c \mid  W_{U,Z}\right) \le 2^{2^r}(1-p^{2^r })^{n-r},$$ 
where $A$ runs over subcomplexes of $Z$ (the number of such subcomplexes is clearly bounded above by $2^{2^r}$).

Since $\mathcal N_U = \sqcup_Z \mathcal N_{U,Z}$ and $\mathcal W_U = \sqcup_Z \mathcal W_{U,Z}$ 
we obtain  
$$P(\mathcal N_U)\le \max_Z P(\mathcal N_{U,Z}\mid W_{U,Z}) \le 2^{{2^r}} (1-p^{2^r})^{n -r}. $$
And finally, we obtain the following upper bound for the probability of the set 
$\mathcal N=\cup_U \mathcal N_{U}$ 
of all non-empty simplicial subcomplexes of $\Delta_n$ which are not $r$-ample: 
\begin{eqnarray}
P(\mathcal N) &\le& \sum_{j=1}^r {\binom n j} \cdot 2^{2^{r}}\left(1-p^{2^r }\right)^{n-r}\nonumber
\\  
&\le &
n^r\cdot 2^{2^{r}}\left(1-p^{2^r }\right)^{n-r}.\label{notample}\label{estima}
\end{eqnarray}
 Clearly, for $n\to \infty$ the 
 expression
 (\ref{estima}) tends to zero. Note also that the probability of the empty simplicial complex is bounded above by $(1-p)^n$ and tends to $0$.
 This completes the proof. 
\end{proof}

\begin{remark}\label{rem32a}
The above arguments prove that the conclusion of Proposition \ref{randoma} holds under a slightly weaker assumption, namely $p_\sigma\in [p, 1-p]$ where $p=p(n)>0$ satisfies
$$p^{{2^r }} = \frac{r\ln n +\mu}{n}$$
with $\mu=\mu(n)$ an arbitrary sequence tending to $\infty$. Examples satisfying the above condition are $p=1/n^\alpha$ with $\alpha\in (0, 2^{-r})$ and $p=1/\ln n$, the latter choice works for any $r$. 
\end{remark}

\begin{remark} \label{rm53}
The arguments of the proof of Proposition \ref{randoma} work without any change if one alters the medial regime assumption by requiring that $p_v=1$ 
for every vertex $v\in [n]$ while $p_\sigma\in [p, 1-p]$ for $\dim \sigma>0$. Formula (\ref{defprob}) implies that in this case the probability measure is supported on the set of simplicial complexes $Y\subset \Delta_n$ with $V(Y)=[n]$, i.e. having 
exactly $n$ vertexes. This observation will be used below in the Proof of Proposition \ref{proba}
\end{remark}

\begin{prop}
\label{proba}
For every $r \geq 5$ and for every $n \ge  r2^r2^{2^{r}}$, there exists an $r$-ample simplicial complex having exactly $n$ vertexes.
\end{prop}
\begin{proof} The expression (\ref{estima}) is an upper bound of the probability that a medial regime random complex on $n$ vertexes 
is not $r$-ample. Clearly, if for some $n$ the RHS of (\ref{estima}) is smaller than 1 then an $r$-ample complex exists.  
The expression (\ref{estima}) is bounded above by 
$$n^r 2^{2^r}e^{-p^{2^r}\cdot(n-r)}= n^r e^{-np^{2^r}}\cdot 2^{2^r}e^{rp^{2^r}}$$
and taking the logarithm we obtain the following inequality 
\begin{eqnarray}\label{ineq3}
np^{2^r} -r\ln n > 2^r \ln 2+ rp^{2^r}
\end{eqnarray}
which guarantees the existence of an $r$-ample complex on $n$ vertexes. Below we shall set  $p=1/2$.
The function $n\mapsto np^{2^r} -r\ln n$ is monotone increasing for $n>r2^{2^r}$ and therefore we only need to show that (\ref{ineq3}) is satisfied for $n=r2^r2^{2^r}.$
The inequality (\ref{ineq3}) turns into 
$$r2^r(1-\ln 2) -r^2 \ln 2 -r\ln r> 2^r \ln 2+ r2^{-2^r}$$
which is equivalent to  
\begin{eqnarray}\label{ineq4}
r(1-\ln 2) -\ln 2 > \frac{r^2 \ln 2 +r\ln r}{2^r} + r2^{-2^r-r}. 
\end{eqnarray}
Given that $\ln 2\simeq 0.6931$ it is easy to see that (\ref{ineq4}) is satisfied for any $r\ge 5$. 
\end{proof}

%

\begin{remark}
\label{toosmall}
Even though a random simplicial complex with $2^{\Omega(2^r)}$ vertices is $r$-ample (as Proposition \ref{proba} claims), $2^{O(2^r/\sqrt{r})}$ vertices do not suffice; this follows from Corollary \ref{cor26} and formula (\ref{frml25}). 
%
\end{remark}
As a byproduct, we also obtain the following result about  2-connectivity of random simplicial complexes in the medial regime. 
\begin{cor}
Every medial regime random simplicial complex  \cite{farber2020random} is $2$-connected, asymptotically almost surely.
\end{cor}
\begin{proof}
This follows from Proposition \ref{randoma} and Theorem \ref{2conn}. 
\end{proof}
Connectivity and simple connectivity of the medial regime random simplicial complexes and vanishing of the Betti numbers were shown in \cite{farber2020random}; the vanishing of the second homotopy group was not previously known.

\section{Explicit construction of ample complexes}
\label{6}

The random construction above shows the existence of $r$-ample simplicial complexes for every~$r$. However, it does not tell us how to construct an $r$-ample complex explicitly. In this section we define a deterministic family of complexes that are guaranteed to be $r$-ample.

Our construction uses ideas from number theory, and generalises the classical Paley graph. In Definitions~\ref{Qnp}-\ref{Xnp}, we introduce the \emph{Iterated Paley Simplicial Complex} $X_{n,p}$ on $n$ vertices, for every odd prime power $n$ and odd prime~$p$ dividing $(n-1)$. But first, we state the main theorem of the section.

\begin{theorem}
\label{amplepaley} 
Let $r \in \mathbb{N}$. Every Iterated Paley Simplicial Complex $X_{n,p}$ with $p>2^{2^r+2r}$ and $n>r^2p^{2r}$ is $r$-ample.
\end{theorem}

Theorem~\ref{amplepaley} is proven below, after the definition of $X_{n,p}$. After the proof, we discuss the selection of the prime parameters $n$ and~$p$, so that $r$-ample complexes can be constructed for every~$r$. We prove the following corollary:

\begin{cor}
\label{findxnp}
For every sufficiently large $r$, there exists an $r$-ample Iterated Paley Complex $X_{n,p}$ on $n=2^{(2+o(1))r2^r}$ vertices.
\end{cor}

To summarize, $\exp(\Omega(r\,2^r))$ vertices are sufficient for constructing an $r$-ample complex explicitly, compared to $\exp(\Omega(2^r))$ vertices  probabilistically, by Corollary~\ref{proba}. We do not know how many vertices are really needed for these constructions to be $r$-ample. However, the lack of $r$-ample complexes of size $\exp(O(2^r/\sqrt{r}))$, by Corollary~\ref{lower}, gives a lower bound.  

\medskip

Before proceeding to the definition of Iterated Paley Simplicial Complexes, we discuss the context in which they have emerged. The uninterested reader may skip to Definition~\ref{Qnp} without loss of continuity.

The \emph{Paley graph} and the \emph{Paley tournament} are long known to satisfy the corresponding $r$-ampleness property \cite{graham1971constructive, blass1981paley, bollobas1981graphs}. Their vertices are the elements of a finite field $\mathbb{F}_n$, with an edge from $x$ to~$y$ if and only if $(y-x)$ is a quadratic residue~\cite[in matrix form]{paley1933orthogonal}. More generally, these constructions exhibit numerous important properties typical to their random counterparts, and are accordingly called \emph{pseudorandom} or \emph{quasirandom}~\cite{alon2004probabilistic, lovasz2012large}. However, these provably $r$-ample graphs and tournaments are nearly square the size of those probabilistically shown to be $r$-ample. Understanding such gaps between randomized and explicit solutions is a recurring theme in the study of combinatorial structures and computational complexity.

The most straightforward extension of Paley's graph to higher dimensions is by including a $d$-dimensional face $x_0 x_1 \cdots x_d$ if $x_0+x_1+\dots+x_d$ is a quadratic residue. Hypergraphs with $(d+1)$-edges constructed by this rule are known to possess \emph{some} quasirandom properties \cite{haviland1989pseudo, chung1991quasi, lenz2015poset}. They also yield large cosystoles in simplicial complexes, pertinent to $d$-dimensional coboundary expansion over $\mathbb{F}_2$, by Kozlov and Meshulam, see \cite{kozlov2019quantitative}. However, they fail to be ample. Indeed, if four vertices satisfy $a+b=c+d$, then $abx$ is a face if and only if $cdx$ is a face, hence some extensions of such a foursome are not available. All the explicit constructions considered in the study of quasirandom hypergraphs break down when it comes to $r$-ampleness.

Our new construction combines three generalizations of Paley graphs. First, if $m|(n-1)$ then, rather than quadratic residues in $\mathbb{F}_n$, one may determine adjacency by means of the multiplicative subgroup of $m$th powers and its cosets. Having similar quasirandom properties \cite{kisielewicz2004pseudo, ananchuen2006cubic}, such graphs proved useful in Ramsey theory \cite{clapham1979class, guldan1983new, su2002lower}. They appear also in the classification of graphs with strong symmetries \cite{peisert2001all, li2009homogeneous}. 

Second, instead of defining hyperedges by summing $x_0+\dots+x_d$, one may use the Vandermonde determinant,
$$ \Delta (x_0,\dots,x_d) \;:=\; \prod\limits_{0 \leq i<j \leq d}(x_i-x_j) $$ 
This is an appealing route because such products are compatible with the multiplicative nature of the above subgroups. Hypergraphs produced this way \cite{kocay1992reconstructing, potovcnik2009vertex, gosselin2010vertex} are known to have several nice properties, but not $r$-ampleness.

The final and novel ingredient in our construction is the repeated use of Paley-like motifs. Faces are selected according to certain cosets of $p$-power residues mod $n$, and those cosets in turn correspond to quadratic residues mod~$p$. For this reason, we name such constructions \emph{Iterated Paley}. The need for this duplexity will be clarified in the ampleness proof.

\medskip

We now formalize the foregoing description. For the following set of definitions, fix an odd prime power $n$, an odd prime~$p$ that divides $n-1$, and a primitive element~$g$ in the finite field~$\mathbb{F}_n$.

\begin{definition}
\label{Qnp}
For $n$, $p$, $g$ as above, let
$$ Q_{n,p} \;:=\; \left\{g^{\alpha} \;\mid\; \alpha \equiv \beta^2 \bmod p, \;\text{for}\; \alpha,\beta \in \mathbb{Z} \right\} \;\subset\; \mathbb{F}_n$$
\end{definition}

\begin{remark*}
Since $p|(n-1)$, we have a multiplicative subgroup $H = \langle g^p \rangle$ of index $p$ in $\mathbb{F}_n^{\times} = \langle g \rangle$, and a group isomorphism $\mathbb{F}_n^{\times}/H \to (\mathbb{F}_p,+)$ taking $g H \mapsto 1$. The set $Q_{n,p}$ is the union of $H$-cosets that correspond to quadratic residues mod~$p$. Hence, it contains about half the elements of the field,
$$ \left|Q_{n,p}\right| \;=\; \frac{p+1}{2p}\,(n-1) $$
\end{remark*}

\begin{definition}
\label{Hnp}
The \emph{Iterated Paley Hypergraph} $H_{n,p}$ has $\mathbb{F}_n$ as its vertex set, and a subset $\{x_1, x_2, \dots, x_t\}$ forms a hyperedge if
$$ \prod_{1\le i<j\le t}(x_i-x_j) \;\in\; Q_{n,p} $$
\end{definition}

\begin{remark*}
Note that $(-1)=g^{(n-1)/2}$, and $(n-1)/2\equiv 0 \bmod p$ since $p$ is odd, hence $(-1)\in H = \langle g^p \rangle$. Therefore, the condition in the definition of~$H_{n,p}$ does not depend on the order of the vertices $x_1, x_2, \dots, x_t$. Note also that all $n$ singletons~$\{x\}$ are hyperedges, because $1 = g^0 \in Q_{n,p}$.
\end{remark*}

The Iteretad Paley Hypergraph might not be a simplicial complex, as it is not necessarily closed downward. We thus consider the largest simplicial complex contained in it, defined as follows. 

\begin{definition}
\label{Xnp}
The \emph{Iterated Paley Simplicial Complex} $X_{n,p}$ has $\mathbb{F}_n$ as its vertex set, and a set $\{x_1, x_2, \dots, x_t\}$ forms a simplex if for every subset 
$\{x_{s_1}, x_{s_2}, \dots, x_{s_k}\} \subseteq \{x_1, x_2, \dots, x_t\}$
$$ \prod_{1\le i<j\le k}\left(x_{s_i}-x_{s_j}\right) \;\in\; Q_{n,p} $$
\end{definition}

\begin{remark*}
The definitions of $Q_{n,p}$ and thereby $H_{n,p}$ and $X_{n,p}$ depend on the choice of primitive element $g \in \mathbb{F}_n$. Any other primitive element $h = g^\alpha \in \mathbb{F}_n$ gives the same construction if $\alpha$ is a quadratic residue mod~$p$, and a different one if not. The two constructions are not isomorphic in general. Our results apply to either choice.
\end{remark*}

\begin{remark*}
Note that $H_{n,p}$ and $X_{n,p}$ are invariant under a rather large group of symmetries $\{x \mapsto ax+b \,\mid\, a \in H, b \in \mathbb{F}_n\}$.
\end{remark*}

\medskip

Before proving Theorem \ref{amplepaley}, we sketch the idea of the proof via a simple example: accommodating one 3-ampleness challenge, posed by three vertices $a,b,c \in X = X_{n,p}$. Given $a,b,c$, suppose that we are looking for another vertex $x \in X$ such that, say, $ax,bx,cx,abx,bcx \in X$ and $acx,abcx \not \in X$. 

We find $x$ in two stages. First we decide on three suitable $H$-cosets $ g^{\alpha}H, g^{\beta}H,g^{\gamma}H$, where $H=\langle g^p \rangle$ as before. Then we find $x \in \mathbb{F}_n$ such that $(x-a) \in g^{\alpha}H$, $(x-b) \in g^{\beta}H$, and $(x-c) \in g^{\gamma}H$. Such an $x$ exists by extending the uncorrelation property of squares, from Paley graphs. Specifically, 3 different additive translations of $p$-power cosets must intersect in $n/p^3 \pm O(\sqrt{n})$ elements. 

Without knowing better, we pick $\alpha, \beta, \gamma \in \mathbb{F}_p$ one by one. The requirement $ax \in X$ implies that $\alpha$ must be a square, which gives $\lceil p/2 \rceil$ options. A short calculation shows that $bx, abx \in X$ require both $\beta$ and $\beta+\delta$ to be squares, where $\delta$ is determined by $(a-b)g^{\alpha} \in g^{\delta}H$. This has $p/2^2\pm O(\sqrt{p})$ solutions by the same ampleness property of Paley graphs. The requirements $cx,bcx \in X$ and $acx,abcx \not \in X$ give four constraints: $\gamma$ and $\gamma + \delta_b$ are squares while $\gamma + \delta_a$ and $\gamma + \delta_{ab}$ are nonsquares, where $\delta_a,\delta_b,\delta_{ab}$ are known from $a,b,c,\alpha,\beta$. This is satisfied by $p/2^4\pm O(\sqrt{p})$ elements of~$\mathbb{F}_p$, by the same reason that Paley graphs are 4-ample.

One has to be a bit careful to avoid contradictions between requirements. For example, $\delta_a=\delta_b$ might mean no solution for~$\gamma$. The proof will avoid such problematic cases with advance planning. On the other hand, sometimes we can take shortcuts. For example, $acx \not \in X$ makes $abcx \not \in X$ come for free. We will not rely on such considerations, as they would not simplify the argument in general. That makes our proof apply to hypergraphs too. 

\medskip

With the above example in mind, we begin with a formal proof of Theorem~\ref{amplepaley}. We first formalize the idea that at every step we have an abundance of choices for the witness to ampleness, with differences lying in the necessary cosets. The following lemma generalizes a property of Paley graphs and tournaments~\cite{graham1971constructive, bollobas1981graphs, blass1981paley} to characters of order~$m$.  

\begin{lemma}
\label{findx} 
In a finite field $\mathbb{F}_q$, let $A \,{\subset}\, \mathbb{F}_q^{\times}$ be a proper multiplicative subgroup of index $m$. Given $d$ cosets of~$A$, 
$$ A_1, A_2, \dots, A_d \;\in\; \mathbb{F}_q^{\times}/A $$ 
and pairwise distinct elements
$$ c_1,c_2,\dots,c_d \;\in\; \mathbb{F}_q $$
the number of elements $x \in \mathbb{F}_q$ satisfying 
$$ (x-c_1) \in A_1, \;(x-c_2) \in A_2, \;\dots, \; (x-c_d) \in A_d $$ 
is at least
$$ \frac{q}{m^d} \;-\; (d-1)\sqrt{q} \;-\; \frac{d}{m}$$
\end{lemma}

This lemma basically says that different additive translates of cosets of $m$-power residues are ``mutually uncorrelated''. Their intersection is of order $q/m^d$ as expected from random subsets, up to an error term of about~$d\sqrt{q}$. A~proof of Lemma~\ref{findx} is given in Appendix~\ref{appendix}.

\begin{proof}[Proof of Theorem \ref{amplepaley}] 
Let $X = X_{n,p}$ be as in Definition~\ref{Xnp}. Consider a set of vertices $\sigma = \{x_1,\dots,x_d\} \subseteq \mathbb{F}_n = V(X)$. Throughout this proof, $\sigma$ is assumed to be nonempty. The Vandermonde determinant of~$\sigma$ in $\mathbb{F}_n$ falls into one of the $p$ cosets of $H = \langle g^p \rangle$, 
$$ \Delta (\sigma) \;=\; \pm \!\prod_{1 \le i < j \le d}\! (x_i - x_j) \;\in\; g^{\alpha(\sigma)}H $$
The exponent $\alpha(\sigma) \in \{0,1,\dots,p-1\}$ is uniquely determined for every~$\sigma$, because $\pm 1 \in H$. In view of the remark following Definition~\ref{Qnp}, we regard $\alpha(\sigma)$ as an element of $\mathbb{F}_p$. 

Recall that a simplex $\sigma \in X$ if and only if for all~$\tau \subseteq \sigma$ the Vandermonde determinant $\Delta (\tau) \in Q_{n,p}$. That is equivalent to $\alpha(\tau) \in Q_p \cup \{0\}$, where
$$ Q_p \;:=\; \left\{\beta^2 \;\mid\; \beta \in \mathbb{F}_p^{\times}\right\} $$
$Q_p$ is the multiplicative subgroup of quadratic residues mod~$p$. The coset of quadratic nonresidues is denoted by $Q_p^c := \mathbb{F}_p^\times \setminus Q_p$.

To verify that $X$ is $r$-ample, consider a set $U \subset \mathbb{F}_n$ of $r$ vertices, and a subcomplex $Y \subseteq X_U$. We seek a vertex $x \in \mathbb{F}_n \setminus U$ such that for every $\sigma \in X_U$ the simplex $\sigma x \in X$ if and only if~$\sigma \in Y$. Here $\sigma x$ stands for the simplex~$\sigma \cup \{x\}$.

By the above characterization of the simplices of~$X$, it is sufficient for the desired vertex~$x$ to solve the following set of $2^r-1$ constraints,
\begin{equation}
\label{problem}
\tag{$\star$}
\forall\, \sigma \subseteq U, \;\;\;\;\;\;\;\; \alpha(\sigma x) \;\in\; \begin{cases} Q_p & \text{if}\; \sigma \in Y \\ Q_p^c & \text{if}\; \sigma \not\in Y \end{cases}    
\end{equation} 
For every hypergraph $Y$ on every set of $r$ vertices $U \subset \mathbb{F}_n$, we show that this problem is indeed satisfiable.

We rewrite the above constraints on $x$. Clearly, the Vandermonde determinant of $\sigma x$ decomposes as follows
$$ \Delta(\sigma x) \;=\; \pm \!\prod_{1 \le i < j \le d}\! (x_i - x_j) \, \prod_{v \in \sigma} (x - v) \;=\; \pm \Delta(\sigma) \prod_{v \in \sigma} \Delta(vx) $$
Applying the quotient map $\mathbb{F}_n^{\times} \to \mathbb{F}_n^{\times} / H \xrightarrow{\sim} \mathbb{F}_p^+$ such that $gH \mapsto 1$ yields the following congruence in~$\mathbb{F}_p$,
$$ \alpha(\sigma x) \;\equiv\; \alpha(\sigma) + \sum_{v \in \sigma} \alpha(vx) $$

We introduce $r$ new variables, $\xi_v \in \mathbb{F}_p$ for each $v \in U$, related to $x$ via $\xi_v = \alpha(vx)$, and obtain an equivalent reformulation of~(\ref{problem}), with the $r+1$ variables $\xi_v \in \mathbb{F}_p$ and $x \in \mathbb{F}_n$.
\begin{align*}
\label{i}\tag{I}
&\forall\, \sigma \subseteq U &&\alpha(\sigma) + \sum_{v \in \sigma} \xi_v \;\in\; \begin{cases} Q_p & \text{if}\; \sigma \in Y \\ Q_p^c & \text{if}\; \sigma \not\in Y \end{cases} \\[0.25em]
\label{ii}\tag{II}
&\forall\, v \in U && (x - v) \;\in\; g^{\xi_v} H
\end{align*}

We now show that given any assignment to the $r$ variables $\{\xi_v\}$ there exists $x \in \mathbb{F}_n$ that satisfies~(\ref{ii}). Indeed, applying Lemma~\ref{findx} with $q=n$, $A=H$, $m=p$, and $d=r$, the number of $x \in \mathbb{F}_n$ satisfying $(x-v) \in g^{\xi_v}H$ for every $v \in U$ is at least
$$\frac{n}{p^r} \;-\; (r-1)\sqrt{n} \;-\; \frac{r}{p} $$ 
Since $n > r^2p^{2r}$, this lower bound is positive, and there exists at least one such solution $x \in \mathbb{F}_n  \setminus U$. This reduces the problem to finding $\{\xi_v\}$ that satisfy~(\ref{i}).

Let $U = \{u_1,\dots,u_r\}$ in arbitrary order. Given $\sigma \subseteq U$, its last vertex $u_i$ in this sequence is called the \emph{top vertex} of $\sigma$. To be precise, $u_i \in \sigma$ and $u_j \not \in \sigma$ for $j > i$. Since the constraints in~(\ref{i}) are labeled by $\sigma \subseteq U$ and include the variables are $\xi_v$ for $v \in \sigma$, we determine $\xi_{u_1}, \dots, \xi_{u_r}$ inductively, selecting each $\xi_{u_i}$ according to the constraints where $u_i$ is the top vertex. We abbreviate $\xi_i = \xi_{u_i}$ as no confusion can arise. 

Supposing $\xi_{1},\dots,\xi_{{i-1}}$ are determined, the next variable $\xi_{i}$ has to satisfy the $2^{i-1}$ constraints where $u_i$ is the top vertex.
\begin{equation}
\label{xiui}
\tag{$\star\star$}
\forall\;\, \{u_i\} \subseteq \sigma \subseteq \{u_1,\dots,u_i\}, \;\;\;\;\;\;\;\;  \xi_{i} + \delta(\sigma) \;\in\; \begin{cases} Q_p & \text{if}\; \sigma \in Y \\ Q_p^c & \text{if}\; \sigma \not\in Y \end{cases}
\end{equation}
where
$$ \delta(\sigma) \;:=\; \alpha(\sigma) + \sum_{v \in \sigma \setminus u_i} \xi_v \;\in\; \mathbb{F}_p $$
is already known from the variables determined earlier than~$\xi_i$.

Before invoking Lemma~\ref{findx} to show that $\xi_{i}$ exists, one has to make sure that all $\delta(\sigma)$ in~(\ref{xiui}) are distinct. This requires some care when selecting $\xi_{1},\dots,\xi_{{i-1}}$. Suppose that $u_i \in \sigma \cap \sigma'$ is the common top vertex of $\sigma$ and~$\sigma'$, and $u_j \in \sigma \setminus \sigma'$ is the top vertex of their difference $\sigma \triangle \sigma' = (\sigma \setminus \sigma') \cup (\sigma' \setminus \sigma)$. Then the condition $\delta(\sigma) \neq \delta(\sigma')$ takes the form
$$ \xi_{j} \;\neq\; \left(\alpha(\sigma') + \sum_{v \in (\sigma' \setminus \sigma)} \xi_{v}\right) - \left( \alpha(\sigma) + \!\!\sum_{v \in (\sigma \setminus \sigma') \setminus u_j} \!\!\xi_{v}\right) $$
Since the value on the right hand side is known when selecting $\xi_{j}$ it can be avoided as long as there are enough other options supplied by Lemma~\ref{findx}. The number of forbidden values for $\xi_{j}$ is at most the number of such pairs of simplices, $\{\sigma, \sigma'\}$, with common top vertex $u_i$ and top ``differentiating'' vertex~$u_j$. The number of these pairs is
$$ \left(2^{j-1}\right)^2 (2^{r-j} - 1) \;<\; 2^{r+j-2} $$
To sum up, in order to enable solutions for all $\xi_{i}$ we will actually find $2^{r+j-2}$ potential solutions for every variable $\xi_{j}$, and proceed with one that evades collisions among~$\delta(\sigma)$ as shown above.

We thus assume by induction that $\xi_1,\dots,\xi_{i-1}$ are given and the $2^{i-1}$ constraints in~(\ref{xiui}) have distinct translations $\delta(\sigma)$, and apply Lemma~\ref{findx} with $q=p$, $A=Q_p$, $m=2$, and $d=2^{i-1}$. This guarantees at least 
$$ \frac{p}{2^{2^{i-1}}} \;-\; \left({2^{i-1}}-1\right) \sqrt{p} \;-\; \frac{{2^{i-1}}}{2} $$
possible values for the variable~$\xi_{i}$.
Since $p > 2^{2^r + 2r}$ and $i \le r$, this number is greater than
$$ 2^{2^{r-1}+2r}-(2^{r-1}-1)2^{2^{r-1}+r} - 2^{r-2} \;\geq\; 2^{2^{r-1}+2r-1} \;\geq\; 2^{r+i-2} $$
for all $i \leq r$, as required. 

In conclusion, there exist $\xi_{1},\dots,\xi_{r} \in \mathbb{F}_p$ satisfying~(\ref{xiui}) for every~$i$, and hence also~(\ref{i}). As shown above, this yields a vertex $x \in \mathbb{F}_n$ that satisfies~(\ref{ii}) and hence $lk_X(x) \cap X_U = Y$ as required.
\end{proof}

In the rest of this section, we discuss the selection of parameters $n$ and $p$ for $r$-ample Iterated Paley Simplicial Complexes.

The construction requires two primes satisfying $n \equiv 1 \bmod p$, that are large enough as in Theorem~\ref{amplepaley}. Given a prime $p$, the existence of arbitrarily large primes $n \in p\mathbb{N}+1$ is a special case of the classical Dirichlet Theorem. This case actually follows from an elementary argument. For $N > p$, let $n$ be a prime divisor of $M = 1 + N! + \dots + (N!)^{p-1} = ((N!)^p-1)/(N!-1)$. Since $n|((N!)^p-1)$, necessarily~$n>N$. If $N! \equiv_n 1$ then $M \equiv_n p$, which is ruled out by $n|M$. Therefore, $N! \not\equiv_n 1$ while $(N!)^p \equiv_n 1$. By Lagrange's theorem, $p|(n-1)$, as desired.

However, in order to establish our quantitative result, Proposition~\ref{findxnp}, we need a prime $n$ roughly of order $p^{2r}$. Dirichlet's theorem asserts that about $1/(p-1)$ of all primes are contained in the arithmetic progression $p\mathbb{N}+1$, in an appropriate sense of asymptotic density~\cite{iwaniec2004analytic}. The following lemma uses quantitative estimates of the ``error term'' to bound the gaps between these primes, which provides such a prime $n$ that is not too large.

\begin{lemma}
\label{p8}
There exists a constant $P$, such that for every prime $p>P$ and every $M \geq p^8$ there exists a prime $n \equiv 1\bmod p$ in the interval
$$ M \;<\; n \;<\; \sqrt{p} \,M $$
\end{lemma}

\begin{proof}
For $a$ and~$q$ coprime, the number of primes less than or equal to~$x$ that are congruent to~$a \bmod q$ is denoted
$$\pi(x; q, a) \;=\; \left|\{n\leq x \;:\; n\text{ is prime, } n \equiv a\bmod q\}\right|$$
Bounds on this number under various assumptions on the relation between $q$ and $x$ are given by the Brun--Titchmarsh theorem and the Siegel--Walfisz theorem. By recent improved bounds due to Maynard \cite[Thms~1~\&~2]{JamesMaynard2013}, there exist effectively computable positive constants $Q$ and~$R$, such that if $q\geq Q$ and $x\geq q^8$ then
\begin{align*}
\frac{R\,\log q}{\sqrt{q}}\cdot \frac{x}{\phi(q)\log x} \;<\; \pi(x; q, a) \;<\; \frac{2\mathrm{Li}(x)}{\phi(q)}
\end{align*}
Here $\phi(q) = |\{a < q \,:\, (a,q)=1\}|$ is Euler's totient function, and the function $\mathrm{Li}(x) = \int_2^x \tfrac{dt}{\log t} \sim \tfrac{x}{\log x} $ is the Eulerian logarithmic integral. In fact, $\text{Li}(x) < \tfrac{3x}{2\log x}$ will suffice for our needs. 

Letting $q=p$ and $a=1$, it follows for any $p > P = \max(Q, \exp(4/R))$ and $M > p^8$ that
$$ \pi(M;p,1) \;<\; \frac{3M}{(p-1) \log M} \;<\; \frac{R\,\log p\,\sqrt{p}\,M}{\sqrt{p}\,(p-1) \log (\sqrt{p}\,M)} \;<\; \pi(\sqrt{p}M;p,1) $$
The middle inequality is verified by
$$ \frac{3 \log (\sqrt{p}\,M)}{R\,\log p \, \log M} \;<\; \frac{3}{R\,\log p}\left(\frac18 + 1\right) \;<\; 1 $$
Since $\pi(M;p,1) < \pi(\sqrt{p}M;p,1)$, there exists at least one prime $n \equiv 1 \bmod p$ between $M$ and $\sqrt{p}M$, as required.
\end{proof}

\begin{proof}[Proof of Proposition \ref{findxnp}]
We now show there exist parameters satisfying the assumptions $p > 2^{2^r+2r}$ and $n > r^2 p^{2r}$ of Theorem \ref{amplepaley}, and $n = 2^{(2+o(1))r2^r}$.

In selecting~$p$, we can just rely on \emph{Bertrand's postulate}, i.e., for every $N \in \mathbb{N}$ there exists a prime between $N$ and~$2N$. Thus, there exists a prime~$p$ in the range
$$ 2^{2^r+2r} \;<\; p \;<\; 2^{2^r+2r+1} $$
Suppose that $r$ is large enough so that $p$ satisfies Lemma~\ref{p8}. We pick a prime $n \equiv 1 \bmod p$ in the range
$$ r^2p^{2r} \;<\; n \;<\; r^2p^{2r+\tfrac12} $$
Therefore, for $r$ sufficiently large, there exists an $r$-ample Iterated Paley Simplicial Complex $X_{n,p}$, on at most 
$$ n \;<\; r^2 2^{(2^r+2r+1)(2r+\frac{1}{2})} \;=\; 2^{2r2^r(1+o(1))} $$
vertices.
\end{proof}

\begin{remark}
\label{explicit}
Finally, we note that the construction of $X_{n,p}$ is \emph{explicit} at least in the following sense. 

Given $r \in \mathbb{N}$, one can find suitable primes $p$ and $n = \exp(O(r2^r))$ and a primitive $g \in \mathbb{F}_n$ in $\text{poly}(n)$ time. One can also decide whether a given face belongs to the $r$-dimensional skeleton of $X_{n,p}$ in $\text{poly}(n)$ time. 

These rough estimates leave some room for improvement, as the description of $X_{n,p}$ and such a face are in fact poly-logarithmic in~$n$. Assuming the generalized Riemann hypothesis, or heuristics on the distribution of primes, may help in selecting $p$, $n$, and~$g$. We do not pursue this matter here.
\end{remark}

\appendix
\section{Proof of Lemma \ref{findx}}
\label{appendix}

Following works on Paley graphs and tournaments \cite{graham1971constructive, bollobas1981graphs,blass1981paley}, we use character sums to prove Lemma~\ref{findx}. A~multiplicative \emph{character} of a finite field $\mathbb{F}_q$ is a map $\chi:\mathbb{F}_q \to \mathbb{C}$, such that $\chi(0)=0$, $\chi(1)=1$, and $\chi(ab)=\chi(a)\chi(b)$ for every $a,b \in \mathbb{F}_q$. Since $\chi$ is a homomrphism between the multiplicative groups, its image is all $m$th roots of unity, where $m = (q-1)/|\ker \chi|$ is called the \emph{order} of~$\chi$. 

The following estimate of character sums is based on the work of Andr\'e Weil~\cite{burgess1962character, schmidt1976equations}. This formulation appears in \cite[Thm~5.41]{lidl1997finite} or \cite[Thm~11.23]{iwaniec2004analytic}.

\begin{theorem}[Weil]
Let $\chi$ be a character of order $m > 1$ of a finite field~$\mathbb{F}_q$, and let $f(x)$ be a polynomial over $\mathbb{F}_q$, that cannot be written as an $m$th power, $c\cdot g(x)^m$. If $f(x)$ has at most $d$ distinct roots in a splitting field, then
$$ \left|\sum_{x \in \mathbb{F}_q} \chi\left(f(x)\right) \right| \;\leq\; (d-1)\sqrt{q} $$
\end{theorem}

\begin{proof}\emph{(Lemma~\ref{findx})}
Let $\alpha$ be a primitive element in $\mathbb{F}_q$, and let $\omega = e^{2 \pi i / m}$. In terms of the subgroup $A$ of index $m$, we define the multiplicative order-$m$ character $\chi(x) = \omega^{t}$ for every $x \in \alpha^{t}A$ and $t \in \mathbb{Z}_m = \mathbb{Z}/m\mathbb{Z}$, and as usual set $\chi(0)=0$. 

Let $A_1,\dots,A_d$ be $A$-cosets, and $c_1,\dots,c_d$ be distinct field elements, as in the lemma. We define $t_1, \dots, t_d \in \mathbb{Z}_m $ such that $A_i = \alpha^{t_i}A$, and consider the function
$$ S(x) \;=\; \prod\limits_{i=1}^{d} \left( \sum\limits_{j = 0}^{m-1} \left(\omega^{-t_i}\chi(x-c_i)\right)^{j} \right) $$
If $x$ satisfies $(x-c_i) \in \alpha^{t_i}A$ then $\chi(x-c_i) = \omega^{t_i}$ and the $i$th factor equals~$m$. Otherwise it
is the sum of all $m$th roots of unity and therefore vanishes, except in the case $x=c_i$ where it contributes~$1$. It follows that $S(x)=m^d$ for every $x$ that is counted in the lemma. Any other $x$ attains $S(x)=0$, apart from $x \in \{c_1,\dots,c_d\}$ where $|S(x)| \leq m^{d-1}$.

In conclusion, if $N$ is the number of $x \in \mathbb{F}_q$ that satisfy $(x-c_i) \in \alpha^{t_i}A = A_i$ for all $i \in \{1,\dots,d\}$, then
$$ \left|\sum\limits_{x \in \mathbb{F}_q} S(x) \right| \;\leq\; N m^d + d m^{d-1} $$

On the other hand, we expand the same sum over $S(x)$ into $m^d$ different sums of characters.
\begin{align*}
\sum\limits_{x \in \mathbb{F}_q} S(x) \;&=\; \sum\limits_{x \in \mathbb{F}_q} \, \prod\limits_{i=1}^{d} \left( \sum\limits_{j = 0}^{m-1} \omega^{-j t_i} \chi(x-c_i)^j \right) \\
\;&=\; \sum\limits_{x \in \mathbb{F}_q} \sum_{j_1 \dots j_d} \left(\prod\limits_{i=1}^d\omega^{-j_i t_i}\right) \; \prod\limits_{i=1}^d \chi(x-c_i)^{j_i} \\
\;&=\; \sum_{j_1 \dots j_d} \omega^{-\sum_i j_i \xi_i} \sum\limits_{x \in \mathbb{F}_q} \; \chi\left(\prod\limits_{i=1}^d (x-c_i)^{j_i} \right)
\end{align*}
The first term, which corresponds to $(j_1,\dots,j_d) = (0,\dots,0)$, is equal to~$q$. Recall that $c_1,\dots,c_d$ are distinct and $\max_ij_i < m$. By Weil's Theorem, it follows that each one of the other $m^d-1$ terms is bounded in absolute value by~$(d-1)\sqrt{q}$. Therefore, by the triangle inequality, 
$$ \left|\sum\limits_{x \in \mathbb{F}_q} S(x) \right| \;\geq\; q - (m^d-1)(d-1)\sqrt{q} $$
The lemma now follows by combining the two estimates of the sum, and solving for~$N$.
\end{proof}

{
\bibliographystyle{alpha}
\bibliography{ample3}

\newcommand{\etalchar}[1]{$^{#1}$}
\begin{thebibliography}{BCI{\etalchar{+}}20}

\bibitem[AC06]{ananchuen2006cubic}
Watcharaphong Ananchuen and Lou Caccetta.
\newblock Cubic and quadruple paley graphs with the ne. c. property.
\newblock {\em Discrete mathematics}, 306(22):2954--2961, 2006.

\bibitem[AS04]{alon2004probabilistic}
Noga Alon and Joel~H Spencer.
\newblock {\em The probabilistic method}.
\newblock John Wiley \& Sons, 2004.

\bibitem[BCI{\etalchar{+}}20]{battison}
Federico Battiston, Giulia Cencetti, Iacopo Iacopini, Vito Latora, Maxime
  Lucas, Alice Patania, Jean-Gabriel Young, and Giovanni Petri.
\newblock Networks beyond pairwise interactions; structure and dynamics.
\newblock {\em Physics Reports}, 2020.

\bibitem[BEH81]{blass1981paley}
Andreas Blass, Geoffrey Exoo, and Frank Harary.
\newblock Paley graphs satisfy all first-order adjacency axioms.
\newblock {\em Journal of Graph Theory}, 5(4):435--439, 1981.

\bibitem[BH79]{blass1979properties}
Andreas Blass and Frank Harary.
\newblock Properties of almost all graphs and complexes.
\newblock {\em Journal of Graph Theory}, 3(3):225--240, 1979.

\bibitem[Bol85]{bollobas1985random}
B{\'e}la Bollob{\'a}s.
\newblock {\em Random graphs}.
\newblock Number~73. Cambridge university press, 1985.

\bibitem[Bon09]{bonato2009search}
Anthony Bonato.
\newblock The search for ne. c. graphs.
\newblock {\em Contributions to Discrete Mathematics}, 4(1), 2009.

\bibitem[BT81]{bollobas1981graphs}
B{\'e}la Bollob{\'a}s and Andrew Thomason.
\newblock Graphs which contain all small graphs.
\newblock {\em European Journal of Combinatorics}, 2(1):13--15, 1981.

\bibitem[Bur62]{burgess1962character}
DA~Burgess.
\newblock On character sums and primitive roots.
\newblock {\em Proceedings of the London Mathematical Society}, 3(1):179--192,
  1962.

\bibitem[Cam97]{cameron1997random}
Peter~J Cameron.
\newblock The random graph.
\newblock In {\em The Mathematics of Paul Erd{\H{o}}s II}, pages 333--351.
  Springer, 1997.

\bibitem[CEV85]{caccetta1985property}
Louis Caccetta, Paul Erdos, and Kaipillil Vijayan.
\newblock A property of random graphs.
\newblock {\em Ars Combin}, 19:287--294, 1985.

\bibitem[CG91]{chung1991quasi}
Fan~RK Chung and Ronald~L Graham.
\newblock Quasi-random set systems.
\newblock {\em Journal of the American Mathematical Society}, 4(1):151--196,
  1991.

\bibitem[Che93]{cherlin1993combinatorial}
Gregory~L Cherlin.
\newblock Combinatorial problems connected with finite homogeneity.
\newblock {\em Contemporary Mathematics}, 131:3--3, 1993.

\bibitem[Cla79]{clapham1979class}
CRJ Clapham.
\newblock A class of self-complementary graphs and lower bounds of some ramsey
  numbers.
\newblock {\em Journal of Graph Theory}, 3(3):287--289, 1979.

\bibitem[ER63]{erdHos1963asymmetric}
Paul Erd{\H{o}}s and Alfr{\'e}d R{\'e}nyi.
\newblock Asymmetric graphs.
\newblock {\em Acta Mathematica Hungarica}, 14(3-4):295--315, 1963.

\bibitem[Fag76]{fagin1976probabilities}
Ronald Fagin.
\newblock Probabilities on finite models 1.
\newblock {\em The Journal of Symbolic Logic}, 41(1):50--58, 1976.

\bibitem[FM20]{farber2020random}
Michael Farber and Lewis Mead.
\newblock Random simplicial complexes in the medial regime.
\newblock {\em Topology and its Applications}, page 107065, 2020.

\bibitem[FMNar]{farber2019lower}
Michael Farber, Lewis Mead, and Tahl Nowik.
\newblock Random simplicial complexes, duality and the critical dimension.
\newblock {\em Journal of Topology and Analysis}, to appear.

\bibitem[FMS19]{farber2019rado}
Michael Farber, Lewis Mead, and Lewin Strauss.
\newblock The rado simplicial complex.
\newblock {\em arXiv preprint arXiv:1912.02515}, 2019.

\bibitem[Gos10]{gosselin2010vertex}
Shonda Gosselin.
\newblock Vertex-transitive self-complementary uniform hypergraphs of prime
  order.
\newblock {\em Discrete mathematics}, 310(4):671--680, 2010.

\bibitem[GS71]{graham1971constructive}
Ronald~L Graham and Joel~H Spencer.
\newblock A constructive solution to a tournament problem.
\newblock {\em Canadian Mathematical Bulletin}, 14(1):45--48, 1971.

\bibitem[GT83]{guldan1983new}
Filip Guldan and Pavel Tomasta.
\newblock New lower bounds of some diagonal ramsey numbers.
\newblock {\em Journal of Graph Theory}, 7(1):149--151, 1983.

\bibitem[HT89]{haviland1989pseudo}
Julie Haviland and Andrew Thomason.
\newblock Pseudo-random hypergraphs.
\newblock In {\em Annals of Discrete Mathematics}, volume~43, pages 255--278.
  Elsevier, 1989.

\bibitem[IK04]{iwaniec2004analytic}
Henryk Iwaniec and Emmanuel Kowalski.
\newblock {\em Analytic number theory}, volume~53.
\newblock American Mathematical Soc., 2004.

\bibitem[KM75]{kleitman1975}
D.~Kleitman and G.~Markowsky.
\newblock On dedekind's problem: The number of isotone boolean functions. ii.
\newblock {\em Transactions of the American Mathematical Society},
  213:373--390, 1975.

\bibitem[KM19]{kozlov2019quantitative}
Dmitry~N Kozlov and Roy Meshulam.
\newblock Quantitative aspects of acyclicity.
\newblock {\em Research in the Mathematical Sciences}, 6(4):33, 2019.

\bibitem[Koc92]{kocay1992reconstructing}
William Kocay.
\newblock Reconstructing graphs as subsumed graphs of hypergraphs, and some
  self-complementary triple systems.
\newblock {\em Graphs and Combinatorics}, 8(3):259--276, 1992.

\bibitem[KP04]{kisielewicz2004pseudo}
Andrzej Kisielewicz and Wojciech Peisert.
\newblock Pseudo-random properties of self-complementary symmetric graphs.
\newblock {\em Journal of Graph Theory}, 47(4):310--316, 2004.

\bibitem[LLP09]{li2009homogeneous}
Cai~Heng Li, Tian~Khoon Lim, and Cheryl~E Praeger.
\newblock Homogeneous factorisations of complete graphs with edge-transitive
  factors.
\newblock {\em Journal of Algebraic Combinatorics}, 29(1):107--132, 2009.

\bibitem[LM15]{lenz2015poset}
John Lenz and Dhruv Mubayi.
\newblock The poset of hypergraph quasirandomness.
\newblock {\em Random Structures \& Algorithms}, 46(4):762--800, 2015.

\bibitem[LN97]{lidl1997finite}
Rudolf Lidl and Harald Niederreiter.
\newblock {\em Finite fields}, volume~20.
\newblock Cambridge university press, 1997.

\bibitem[Lov12]{lovasz2012large}
L{\'a}szl{\'o} Lov{\'a}sz.
\newblock {\em Large networks and graph limits}, volume~60.
\newblock American Mathematical Soc., 2012.

\bibitem[May13]{JamesMaynard2013}
James Maynard.
\newblock On the brun-titchmarsh theorem.
\newblock {\em Acta Arithmetica}, 157(3):249--296, 2013.

\bibitem[Pal33]{paley1933orthogonal}
Raymond~EAC Paley.
\newblock On orthogonal matrices.
\newblock {\em Journal of Mathematics and Physics}, 12(1-4):311--320, 1933.

\bibitem[Pei01]{peisert2001all}
Wojciech Peisert.
\newblock All self-complementary symmetric graphs.
\newblock {\em Journal of Algebra}, 240(1):209--229, 2001.

\bibitem[P{\v{S}}09]{potovcnik2009vertex}
Primo{\v{z}} Poto{\v{c}}nik and Mateja {\v{S}}ajna.
\newblock Vertex-transitive self-complementary uniform hypergraphs.
\newblock {\em European Journal of Combinatorics}, 30(1):327--337, 2009.

\bibitem[Rad64]{rado1964universal}
Richard Rado.
\newblock Universal graphs and universal functions.
\newblock {\em Acta Arithmetica}, 4(9):331--340, 1964.

\bibitem[Sch76]{schmidt1976equations}
Wolfgang~M Schmidt.
\newblock {\em Equations over finite fields an elementary approach}.
\newblock Springer, 1976.

\bibitem[SLLL02]{su2002lower}
Wenlong Su, Qiao Li, Haipeng Luo, and Guiqing Li.
\newblock Lower bounds of ramsey numbers based on cubic residues.
\newblock {\em Discrete Mathematics}, 250(1-3):197--209, 2002.

\bibitem[Spe93]{spencer1993zero}
Joel Spencer.
\newblock Zero-one laws with variable probability.
\newblock {\em The Journal of symbolic logic}, 58(1):1--14, 1993.

\bibitem[Win93]{winkler1993random}
P~Winkler.
\newblock Random structures and zero-one laws, finite and infinite
  combinatorics in sets and logic, nw sauer, re woodrow and b. sands, eds.,
  nato advanced science institutes series, 1993.

\end{thebibliography}
}

\end{document}